\documentclass[11pt]{article}
\usepackage{amsmath,amsfonts,amssymb}
\usepackage{amsthm}
\usepackage{tikz-cd}
\usepackage{url}

\usepackage{array}
\newcommand\Tstrut{\rule{0pt}{2.6ex}}         
\newcommand\Bstrut{\rule[-1.6ex]{0pt}{0pt}}   

\def\be{\begin{equation}}
\def\ee{\end{equation}}
\def\bea{\begin{eqnarray}}
\def\eea{\end{eqnarray}}
\def\nn{\nonumber}

\newcommand\sLP{\\[\smallskipamount]}
\newcommand\mLP{\\[\medskipamount]}
\newcommand\bLP{\\[\bigskipamount]}

\newcommand\al\alpha
\newcommand\ga\gamma

\newcommand\de\delta
\newcommand\ep\epsilon
\newcommand\la\lambda
\newcommand\si\sigma
\newcommand\half{\frac12}
\newcommand\thalf{\tfrac12}

\newcommand\goE{\mathfrak{E}}
\newcommand\goF{\mathfrak{F}}
\newcommand\goR{\mathfrak{R}}

\newcommand\FSA{\mathcal{A}}
\newcommand\FSC{\mathcal{C}}
\newcommand\FSH{\mathcal{H}}
\newcommand\FSM{\mathcal{M}}
\newcommand\FSO{\mathcal{O}}
\newcommand\FSS{\mathcal{S}}

\newcommand\CC{\mathbb{C}}
\newcommand\ZZ{\mathbb{Z}}

\newcommand\td{\tilde}
\newcommand\iy\infty
\newcommand{\qhyp}[5]{\,\mbox{}_{#1}\phi_{#2}\!\left(
  \genfrac{}{}{0pt}{}{#3}{#4};#5\right)}
\newcommand\const{\mathrm{const.}\,}
\newcommand{\iup}{\mkern1mu\textup{i}\mkern1mu}

\newcommand\Span{\operatorname{Span}}

\numberwithin{equation}{section}

\newtheorem{theorem}{Theorem}[section]
\newtheorem{lm}[theorem]{Lemma}
\newtheorem{prop}[theorem]{Proposition}
\newtheorem{cor}[theorem]{Corollary}
\newtheorem{Definition}[theorem]{Definition}
\newenvironment{df}{\begin{Definition}\rm}{\end{Definition}}
\newtheorem{Remark}[theorem]{Remark}
\newenvironment{remark}{\begin{Remark}\rm}{\end{Remark}}
\newtheorem{Assumption}[theorem]{Assumption}
\newenvironment{ass}{\begin{Assumption}\rm}{\end{Assumption}}

\begin{document}
\title{Automorphisms of the DAHA of type $\check{C_1}C_1$ and\\
non-symmetric
Askey--Wilson functions}

\author{Tom Koornwinder\footnote{\protect\url{thkmath@xs4all.nl},
Korteweg--de Vries Institute, University of Amsterdam,
P.O.~Box 94248, 1090 GE Amsterdam, The Netherlands.}\;
and Marta Mazzocco\footnote{\protect\url{marta.mazzocco@upc.edu},
ICREA, Pg.~Lluis Companys 23, 08010 Barcelona, Spain; 
Departament de Matem\`atiques, Universitat Polit\`ecnica de Catalunya,
08028 Barcelona, Spain; School of Mathematics, University of Birmingham, UK.}
\footnote{Work funded by the Leverhulme Trust Research Project
Grant RPG-2021-047.}}

\date{}

\maketitle

\begin{abstract}
In this paper we consider the automorphisms of the double affine Hecke algebra
(DAHA) of type
$\check{C_1}C_1$ which have a relatively simple
action on the generators and on the parameters, notably a symmetry $t_4$ which
sends the Askey--Wilson (AW) parameters $(a,b,c,d)$ to $(a,b,qd^{-1},qc^{-1})$.
We study how these symmetries act on the basic representation and on the
symmetric and non-symmetric AW polynomials and functions.
Interestingly $t_4$ maps AW polynomials to functions.
We take the rank one case of Stokman's Cherednik kernel for $BC_n$ as the
definition of the non-symmetric Askey--Wilson function. From it we derive
an expression as a sum of a symmetric and an anti-symmetric term.
\end{abstract}

\section{Introduction}
This paper marks the beginning of a programme to consider and possibly
classify the symmetries of the double affine Hecke algebra (DAHA) of type
$\check{C_1}C_1$ and of the Zhedanov algebra, to see how the symmetries of
these two algebras are related, and to study the actions of these symmetries
on the basic representation and on the
symmetric and non-symmetric Askey--Wilson (AW) polynomials and functions.

Some symmetries were earlier observed in the literature.
Stokman \cite[\S2.6]{St2003} gives a projective action of the
modular group $\mathbb P SL_2(\mathbb Z)$ on the DAHA operators.
Oblomkov \cite[Prop.~2.4]{Ob} generalizes this by
giving an action of an extension of
$\overline\Gamma:=\mathbb P GL_2(\mathbb Z)$, moreover acting
on the DAHA itself.
Furthermore,  Cramp\'e et al.~\cite[\S4]{Cr-F-Gab-V} and the first author
\cite[\S2.3]{Koo2023} observe an action of $W(D_4)$
(the Weyl group of root system
$D_4$) on the Zhedanov algebra. A natural question is
therefore to understand, also for other DAHA symmetries,
how they act in the basic representation and
whether they map AW polynomials to AW polynomials.
Interestingly, it turns out that
some of them actually map AW polynomials into AW functions, so that it is more
natural to see whether they map AW functions to AW functions.

Our programme has two main motivations.
On the one hand, DAHA symmetries acting on AW polynomials or functions may
imply possibly new functional or contiguous relations
for these special functions.
On the other hand, since the $\check{C_1}C_1$ DAHA can be obtained as
quantisation of the group algebra of the
monodromy group associated to the sixth Painlev\'e equation (PVI)
\cite{Ma2018},
it is natural to ask whether the symmetry group $G$ of PVI lifts to the
$\check{C_1}C_1$ DAHA. 
Okamoto \cite{Ok} proved that $G$ is isomorphic to the affine Weyl group of the
exceptional root system $F_4$. 
The group $G$ contains as subgroups the affine $W(D_4)$,
its extension by the Klein four-group $\mathcal K_4$ and three more
transformations which can be represented by the action of
$\overline\Gamma$.
So finding the action of affine  $W(D_4)$ and
$\overline\Gamma$ on the $\check{C_1}C_1$ DAHA would provide
an important step to answer this question.
Section 3 on DAHA automorphisms was written with such questions in mind.

Then we continue with the study of the actions of DAHA symmetries
on the basic representation (Section 4) and on the
(symmetric) AW polynomials and functions (Section 6).
We consider in particular the symmetry $t_4$ which
sends the AW parameters $(a,b,c,d)$ to $(a,b,qd^{-1},qc^{-1})$.
This symmetry acts on the AW operator $L$ (the second order $q$-difference
operator of which the AW polynomials and functions are
eigenfunctions). However, it does not act on AW polynomials but it does act on
AW functions.
Indeed, the symmetries of $(a,b,c,d)$ that act on the AW polynomials
are the permutations of $(a,b,c,d)$ (hence not $t_4$), while the ones that
act on the AW functions are the permutations of $(a,b,c,qd^{-1})$, which
include $t_4$.
The symmetries of
the AW function can be seen from its explicit expressions as given in
\cite{Koe-St}, but even nicer from Stokman's rank one symmetric Cherednik kernel
\cite[(4.4), (4.7)]{St2002}, since the symmetries are implied there by
the symmetries of the AW polynomials in which the kernel is expanded.

Finally, in Section 7, we turn to non-symmetric AW polynomials and functions,
which no longer have the $b\leftrightarrow c$ symmetry of their symmetric
counterparts. Non-symmetric AW polynomials were treated in great detail by
Noumi \& Stokman \cite{N-St} (see also \cite{Sa2000} and
\cite{Koo2007}). We need them here
in particular because the non-symmetric AW function, as the rank one case
of Stokman's Cherednik kernel \cite{St2003}, has non-symmetric AW polynomials in
its expansion. The symmetries of the non-symmetric AW polynomials then
imply the symmetries of the non-symmetric AW function, including a
$t_4$ symmetry.
Stokman's general case \cite[\S5.4]{St2003} of this kernel
for the nonreduced root system $BC_n$ is an analogue of Cherednik's kernel
\cite[\S5]{Ch} for reduced root systems.

Also in Section 7 we show that the non-symmetric AW function can be
written as a
sum of a symmetric AW function and an anti-symmetric term involving
a symmetric AW function with shifted parameters.
The symmetries can also be read off from this explicit expression: they follow
from the already known symmetries of the symmetric AW function.
Another motivation for deriving this new expression for the non-symmetric
AW function is that expressions of one-variable
non-symmetric special functions as sums of a symmetric and an anti-symmetric term
are well-known for Bessel functions, Jacobi polynomials (for these two
see for instance \cite{Koo-B}) and AW polynomials (see Section 7.1).
In fact, insertion
 in the rank one Cherednik kernel of such a formula for the non-symmetric
 AW polynomials yields our result.

We already mentioned most of the contents of this paper, but here we
summarize them again.
In Section 2 we introduce the DAHA of type $\check{C_1}C_1$ in two different
but equivalent presentations. In Section 3 about DAHA automorphisms we first
discuss the somewhat subtle notion of a symmetry of a (parameter dependent) DAHA.
Then we present
a number of symmetries for the $\check{C_1}C_1$ DAHA and we discuss some of the
groups they generate.
In Section 4 we define the basic
DAHA representation and consider the action of automorphisms, notably of $t_4$,
on it. We also discuss the spherical subalgebra (which is isomorphic to the
Zhedanov algebra at a special scalar value of the Casimir element), its polynomial
representation, and its symmetries. In Section 5 we introduce AW polynomials
and functions and give some of their properties. Next Section 6 deals with
the action of DAHA automorphisms on AW polynomials and functions.
Finally Section 7 presents non-symmetric AW polynomials and the non-symmetric
AW function as defined by Stokman, and we discuss their symmetries.
Then we derive the new expression of the non-symmetric AW function as a sum
of a symmetric and an anti-symmetric term.
We derive the explicit expression in rank one of Stokman's kernel in Appendix A.
In Appendix B we give an alternative proof that the non-symmetric AW function in
our new expression is an eigenfunction of the $Y$ operator (which was for
Stokman's kernel already part of the definition).

As a follow-up to this paper we have two further papers in preparation, to which
we will sometimes refer in this paper. The tentative titles are:
\\[\medskipamount]
[P2] (with D. Dal Martello) Quantum middle convolution.
\\[\medskipamount]
[P3] Further DAHA automorphisms.

\paragraph{Conventions}
For $q$-shifted factorials and $q$-hypergeometric series we use notation as
given in \cite[\S 1.2]{Gas-R}.
In this paper $q\in(0,1)$ will be fixed.
Mostly we suppress the dependence on $q$ as a separate parameter in the
notation.

We will denote $\CC^*:=\CC\backslash\{0\}$.
Omnipresent parameters $a,b,c,d$ will be generic. In particular they are
supposed to be complex, nonzero, and the parameters together with their inverses
should be all distinct. Further constraints will be introduced in
Assumption \ref{K213}.

For $z\in\CC\backslash (-\iy,0]$ the square root
$\sqrt{z}$, continuously depending on $z$, will be uniquely chosen
such that $\sqrt{z}>0$ if $z>0$.

\section{DAHA of type $\check{C_1}C_1$}
We present the DAHA of type $\check{C_1}C_1$ as in our earlier papers
\cite[(3.1)--(3.4)]{Koo2007}, \cite[(1.1)--(1.5)]{Ma2016} and \cite[(84)]{Koo-Ma}.
This is essentially the case $n=1$ of Sahi's presentation \cite[\S3]{Sa1999}.

 \begin{df}
\label{K142}
The DAHA of type $\check{C_1}C_1$, denoted by
$\FSH_{a,b,c,d}(T_1,T_0,Z)=\FSH_{a,b,c,d}$,
is the unital associative $\CC$--algebra  generated by $T_1,T_0,Z$
with relations
\begin{align}
(T_1+ a b)(T_1+1)=0,\label{sahi1}\\
(T_1 Z  +a)(T_1 Z +b)=0,\label{sahi3}\\
(qT_0 Z^{-1}+ c)(qT_0 Z^{-1}+d)=0,\label{sahi4}\\
(T_0+q^{-1} c d)(T_0+1)=0.\label{sahi2}
\end{align}
\end{df}

Note that \eqref{sahi1}, \eqref{sahi3} and \eqref{sahi2} imply that $T_1,T_0,Z$
have two-sided inverses, respectively given by
\be
\begin{split}
&T_1^{-1}=-a^{-1}b^{-1}T_1-(1+a^{-1}b^{-1}),\quad
T_0^{-1}=-qc^{-1}d^{-1}T_0-(1+qc^{-1}d^{-1}),
\\
&\qquad\qquad\qquad\quad Z^{-1}=-a^{-1}b^{-1}T_1ZT_1-(a^{-1}+b^{-1})T_1.
\end{split}
\label{K168}
\ee
Further important elements of $\FSH_{a,b,c,d}$ are
\be
Y:=T_1T_0,\qquad Y^{-1}=T_0^{-1}T_1^{-1},\qquad D:=Y+q^{-1}abcdY^{-1},
\qquad X:=Z+Z^{-1}.
\label{K83}
\ee
The following lemma is well-known. For completeness we sketch a proof.

\begin{lm}\label{K178}
The elements $X$ and $D$ commute with $T_1$.
\end{lm}
\begin{proof}
It follows from \eqref{sahi1} and \eqref{sahi3} that
\begin{equation*}
T_1Z-Z^{-1}T_1=(ab+1)Z^{-1}-(a+b)=ZT_1-T_1Z^{-1}.
\end{equation*}
Hence $T_1$ commutes with $Z+Z^{-1}=X$.
Also, it follows from \eqref{sahi1} and \eqref{sahi2} that
\begin{equation*}
T_1^2T_0-q^{-1}abcdT_0^{-1}=-(ab+1)T_1T_0+ab(q^{-1}cd+1)=
T_1T_0T_1-q^{-1}abcdT_1T_0^{-1}T_1^{-1}.
\end{equation*}
Hence $T_1$ commutes with $T_1T_0+q^{-1}abcdT_0^{-1}T_1^{-1}=D$.
\end{proof}

\begin{cor}\label{K177}
If $f(Z)$ is a \emph{symmetric} Laurent polynomial in $Z$, i.e.,
$f(Z)=f(Z^{-1})$, then $f(Z)$ commutes with $T_1$.
\end{cor}
\begin{proof}
By Lemma \ref{K178} $T_1$ commutes with $X^n=(Z+Z^{-1})^n$ ($n\in\ZZ_{\ge0}$).
It follows by complete induction that $T_1$ commutes with $Z^n+Z^{-n}$
($n\in\ZZ_{>0}$). And $T_1$ commutes with 1.
\end{proof}

In order to compare relations \eqref{sahi1}--\eqref{sahi2}
with an alternative
presentation below of the DAHA of type $\check{C_1}C_1$ it is convenient
to rewrite \eqref{sahi3}, \eqref{sahi4} as the following equivalent relations.
\begin{align}
(Z^{-1}T_1^{-1}+a^{-1})(Z^{-1}T_1^{-1}+b^{-1})=0,\label{sahi3a}\\
(T_0^{-1}Z+ qc^{-1})(T_0^{-1}Z+ qd^{-1})=0.\label{sahi4a}
\end{align}

In {\cite[Theorem 2.22]{N-St}, \cite[Definition 2.1]{Ob}} the DAHA of type
$\check{C_1}C_1$ is defined as an algebra
$H_{k_1,u_1,u_0,k_0}(V_1,\check{V_1},\check{V_0},V_0)=H_{k_1,u_1,u_0,k_0}$ with
generators $V_1$, $\check{V_1}$, $\check{V_0}$, $V_0$, and relations
\bea
\label{daha2}
(V_1-k_1)(V_1+k_1^{-1})=0,\\
\label{daha1}
(\check{V_1}-u_1)(\check{V_1}+u_1^{-1})=0,\\
\label{daha4}
(\check{V_0}-u_0)(\check{V_0}+u_0^{-1})=0,\\
\label{daha3}
(V_0-k_0)(V_0+k_0^{-1})=0,\\
\label{daha5}
V_0 \check{V_0}\check{V_1}V_1=q^{-1/2},
\eea
where  $k_1,u_1,u_0,k_0\in\CC^*$.
Here the Hecke relations \eqref{daha2}--\eqref{daha3} immediately
yield explicit two-sided inverses of the four generators. Relation
\eqref{daha5} remains essentially unchanged under cyclic permutation of
$V_1$, $\check{V_1}$, $\check{V_0}$, $V_0$.

We have given both presentations $\FSH_{a,b,c,d}(T_1,T_0,Z)$ and
$H_{k_1,u_1,u_0,k_0}(V_1,\check{V_1},\check{V_0},V_0)$ for the DAHA
because some automorphisms in Section 3 can be easier defined in the first
presentation and others in the second. But from Section 4 onwards we will
work only with the first presentation.

The algebras $\FSH_{a,b,c,d}$ and $H_{k_1,u_1,u_0,k_0}$,
with parameters related by
\be\label{eq:SO1}
a=u_1 k_1,\quad b=-u_1^{-1} k_1,\quad c=q^{1/2} u_0 k_0 ,\quad
d=-q^{1/2}u_0^{-1} k_0,
\ee
or vice versa by
\be\label{eq:SO1inv}
k_1=\iup\sqrt{ab},\quad
u_1=-\iup\sqrt{a b^{-1}},\quad
u_0=\iup\sqrt{cd^{-1}},\quad
k_0=-\iup \sqrt{q^{-1}cd},
\ee
are isomorphic
by an isomorphism $\phi=\phi_{a,b,c,d}$, explicitly given by
\be\label{eq:SO}
\begin{split}
&\phi(T_1)= k_1 V_1=\iup\sqrt{ab}\,V_1\,,\quad
\phi(T_0)=k_0 V_0=-\iup \sqrt{q^{-1}cd}\,V_0\,,\\
&\phi(Z) = q^{1/2} V_0\check{V_0}\,,
\qquad\qquad\; \phi(Z^{-1})= \check{V_1} V_1\,,
\end{split}
\ee
or vice versa by
\be\label{eq:SOinv}
\begin{split}
&\phi^{-1}(V_1)=-\iup\sqrt{a^{-1}b^{-1}}\,T_1=k^{-1} T_1,\\
&\phi^{-1}(\check V_1)= \iup\sqrt{ab}\,Z^{-1} T_1^{-1}=k_1 Z^{-1} T_1^{-1},\\
&\phi^{-1}(\check V_0)= -\iup q^{-1}\sqrt{cd}\,T_0^{-1} Z=
q^{-1/2} k_0 T_0^{-1} Z,\\
&\phi^{-1}(V_0)=\iup \sqrt{qc^{-1}d^{-1}}\,T_0=k_0^{-1}T_0.
\end{split}
\ee
In \eqref{eq:SO1inv} and \eqref{eq:SOinv} recall our convention about
square roots, by which we also immediately see:
\begin{lm}
\label{K144}
$a,b,c,d,ab,ab^{-1},cd,cd^{-1}$ have real part $>0$ iff
$u_0,-u_1,-k_0,k_1$ have argument $>\pi/4$ and $<3\pi/4$.
\end{lm}

Unless otherwise said we will from now on assume:
\begin{ass}
\mbox{$ab,ab^{-1},cd,cd^{-1}$, and hence also $a,b,c,d$,
have real part $>0$.}
\label{K213}
\end{ass}

\begin{remark}\label{re:alg-isom}
By comparing the Hecke relations \eqref{sahi1}, \eqref{sahi2},
\eqref{sahi3a},
\eqref{sahi4a} on the one hand and
\eqref{daha2}--\eqref{daha3} on the other hand,
we see that the particular choices made in \eqref{eq:SO1}--\eqref{eq:SOinv}
identify $V_1$, $\check{V_1}$, $\check{V_0}$, $V_0$ with the
given constant
multiples of $T_1$, $Z^{-1}T_1^{-1}$, $T_0^{-1}Z$, $T_0$, respectively, in
this particular order, because the cyclic
relation \eqref{daha5} must be equivalent to $Z Z^{-1}=1$.
An obvious freedom left is that we might have simultaneously
performed cyclic permutation of $V_1$, $\check{V_1}$, $\check{V_0}$, $V_0$
and $k_1$, $u_1$, $u_0$, $k_0$.
\end{remark}

\begin{remark}
Sahi \cite{Sa1999} and Stokman \cite{St2002} do not consider (for $n=1$)
DAHA relations of the form \eqref{daha2}--\eqref{daha5}, but they do use
parameters similar to $k_1,u_1,u_0,k_0$. In \cite[(1)]{Sa1999}
our $k_1,u_1,u_0,k_0$ correspond to
$t_1^{1/2},u_1^{1/2},u_0^{1/2},t_0^{1/2}$, and in \cite[(4.7)]{St2002} to
$t_1,u_1,u_0,t_0$.
\end{remark}

\begin{remark}
Analogous to Lemma \ref{K178} we have that $T_0$ commutes with $D$ and with
$q^{-1/2}Z+q^{1/2}Z^{-1}$. This can for instance be proved by starting with
Lemma \ref{K178},
applying $\phi$ to it, performing a cyclic permutation in the parameters and
the generators, and applying $\phi^{-1}$.
\end{remark}

\section{Some DAHA automorphisms}
In this section we first define what we mean by a DAHA automorphism.
Then we introduce the special automorphisms $t_2$, $t_3$, $t_4$, $\si$ and
$\tau$ which will be of special
interest for their possible action on AW polynomials and functions because
they leave $T_1$ invariant. In connection with
$\si$ we introduce dual parameters $\td a,\td b,\td c,\td d$.
The section ends with short discussions of the extended modular group and
the spherical braid group acting on the DAHA.
We postpone a more comprehensive treatment of DAHA automorphisms to a next paper
[P3].

\subsection{Definition of a DAHA automorphism}
More generally we define an automorphism of a parametrized associative
algebra.
\begin{df}\label{K182}
Let $H_a(\{U\})$ be an associative algebra depending on a parameter $a$ in
some set $E$, for instance some open subset of $\CC^m$, and with a finite set
$\{U\}$ of generators satisfying a finite set of relations $\{p_a(\{U\})=0\}$.
An \emph{automorphism} $t$ of the parametrized family $\{H_a(\{U\})\}$ of
associative algebras is a bijection $\hat t\colon E\to E$ together with
isomorphisms $\check t_a\colon H_a(\{U\})\to H_{\hat t(a)}(\{U\})$.
\end{df}

So $\check t_a$ can be defined by specifying $\check t_a(U)$ for all
generators $U$
as an algebraic expression in terms of the generators and $a$. It can be extended
homomorphically to an endomorphism of the free algebra
$\CC[a]\langle\{U\}\rangle$ (so $\check t_a$ should not map $a$ to $\hat t(a)$
in coefficients depending on $a$).
The set of relations $\{p_a(\{U\})=0\}$ for $H_a(\{U\})$
will be mapped by $\check t_a$ to
$\{p_a(\{\check t_a(U)\})=0\}$. This latter set of relations
should be equivalent to the set of relations
$\{p_{\hat t(a)}(\{U\})=0\}$ for $H_{\hat t(a)}(\{U\})$.
Then we have the desired isomorphism
$\check t_a\colon H_a(\{U\})\to H_{\hat t(a)}(\{U\})$.
\mLP\indent
Suppose that we have for all $a\in E$ faithful representations $\pi_a$ and
$\rho_a$ of $H_a(\{U\})$ on vector spaces $V$ and $W$, respectively.
Then it is a
natural question to find the lower horizontal arrow in the following
commutative diagram.
\begin{equation}
\begin{tikzcd}
H_a(\{U\})\arrow[r,"\check t_a"] \arrow[d,"\pi_a"]
& H_{\hat t(a)}(\{U\}) \arrow[d,"\rho_{\hat t(a)}"] \\
L(V)\arrow[r,"\mbox{?}"]
& L(W)
\end{tikzcd}
\label{K216}
\end{equation}
Here $L(V)$ is a suitable space of linear operators on $V$. It is a main theme
in Section \ref{K187} to find this arrow in the case of the basic
representation of our DAHA.
See there \eqref{K169}, \eqref{K173}, \eqref{K179}, \eqref{K215}
for $t=\tau, \tau^{-1}, t_4, \sigma$, respectively.
\mLP\indent
There is an alternative notion of automorphism. First denote $\{H_a(\{U\})\}$
by $H(a,\{U\})$, where the coordinates of $a$ are now considered as additional
generators of the algebra which are in the center of the algebra (a fact which can
be expressed by evident additional relations). We define 
an \emph{automorphism} of $H(a,\{U\})$ as an algebra isomorphism
$\td t\colon H(a,\{U\})\to H(a,\{U\})$ such that $\td t(a)$
only depends on $a$ and not on the $U$.

We can associate an automorphism $\td t$ of the second type with an
automorphism $t$ of the first type. For given $t$ consider
$\check t_{\hat t^{-1}(a)}\colon H_{\hat t^{-1}(a)}(\{U\})\to H_a(\{U\})$.
Then the set of relations
$\{p_{\hat t^{-1}(a)}(\{\check t_{\hat t^{-1}(a)}(U)\})=0\}$
must be equivalent to the set of relations $\{p_a(\{U\})=0\}$ for $H_a(\{U\})$.
This means that we can define an automorphism
$\td t\colon H(a,\{U\})\to H(a,\{U\})$ of the second kind by putting
\[
\td t:=\hat t^{-1}\circ\check t_{\hat t^{-1}(a)}^{-1}=
\check t^{-1}_a\circ \hat t^{-1}.
\]
Here we let $\hat t$ act on the $a$-parametrized free algebra with generators
$\{U\}$ by sending expressions $f(a)g(\{U\})$ to $f(\hat t^{-1}(a))g(\{U\})$.
There results a contravariant functorial relation between automorphsisms
$t$ of the first type and automorphisms $\td t$ of the second type:
\mLP
\begin{tikzcd}
H_a(\{U\})\arrow[dr,"\check t_a"] \arrow[ddd,dashed,"\mathrm{id}"]
\arrow[rr,"(s\circ t)\,\check{}_{\!a}"]
&& H_{\hat s(\hat t(a))}(\{U\}) \arrow[ddd,dashed,"\hat s\circ\hat t"]\\
& H_{\hat t(a)}(\{U\})\arrow[ur,"\check s_{\hat t(a)}"]
\arrow[d,dashed,"\hat t"]&\\
&H(a,\{U\})\arrow[dl,"\td t"] \arrow[dr,phantom]&\\
H(a,\{U\})\arrow[ur,phantom]&&
H(a,\{U\})\arrow[ul,"\td s"] \arrow[ll,"\td t\circ\td s"]
\end{tikzcd}
\mLP
Although the definition of the automorphism of the second type is more elegant and
more intuitive, it is less suitable for our purposes, because we want to consider
the action of automorphisms on special representations.

\begin{remark}
Very special DAHA automorphisms are considered by Sahi \cite[\S4]{Sa1999} and
Oblomkov \cite[Prop.~2.4]{Ob}. Since they only give a kind of definition
of automorphism in a special context, it is difficult to decide for sure
which definition they use. But it seems that Sahi has an automorphism of the
second type and Oblomkov of the first type. However, in their context the
difference between the two definitions is less important because their
automorphism acting on the parameters is an involution and their automorphism
acting on the generators does not depend on the parameters.
\end{remark}
\noindent
\textbf{Notation}\quad
From now on, by abuse of notation, we write for the maps $\hat t$ and $\check t_a$
associated with an automorphism $t$ of $\{H_a(\{U\})\}$ just
$t\colon E\to E$ and $t_a\colon H_a(\{U\})\to H_{t(a)}(\{U\})$.
\bLP\indent
We now use Definition \ref{K182} for automorphisms of the DAHA
$\FSH_{a,b,c,d}$, or equivalently $H_{u_1,k_1,k_0,u_0}$. Let $t$ be
such an automorphism.
We can relate its action on $\FSH_{a,b,c,d}$ with its action on
$H_{u_1,k_1,k_0,u_0}$
by the following commutative diagrams.
\mLP
\begin{minipage}{7cm}
\begin{tikzcd}
(a,b,c,d)\arrow[r,"t"] \arrow[d]
& (a',b',c',d') \arrow[d] \\
(k_1,u_1,u_0,k_0) \arrow[r,"t"]
& (k_1',u_1',u_0',k_0')
\end{tikzcd}
\end{minipage}
\quad
\begin{minipage}{7cm}
\begin{tikzcd}
(T_1,T_0,Z)\arrow[r, "t_{a,b,c,d}"] \arrow[d, "\phi_{a,b,c,d}"]
& (T_1',T_0',Z') \arrow[d, "\phi_{a',b',c',d'}"] \\
\big(V_1,\check{V_1},\check{V_0},V_0 \big)\quad
\arrow[r, "t_{k_1,u_1,u_0,k_0}"]
& \quad\big(V_1',\check{V_1}',\check{V_0}',V_0' \big)
\end{tikzcd}
\end{minipage}
\mLP
In the right vertical arrow in the right diagram we evaluate
$\phi_{a',b',c',d'}(U')=\phi_{a',b',c',d'}(t_{a,b,c,d}(U))$
($U=T_1$, $T_0$ or $Z$)
by applying $\phi_{a',b',c',d'}$ as an algebra homomorphism to the algebraic
expression in $T_0,T_1,Z$ given by $t_{a,b,c,d}(T_U)$.
Here $\phi_{a',b',c',d'}$
only acts on the generators $T_1,T_0,Z$, but leaves scalars unchanged, even if
they depend on $a,b,c,d$.
See later in this section a few worked out examples, for instance for
$t_4$ in Table \ref{tb:okam}.

\subsection{Inversions of parameters $k_1$, $u_1$, $u_0$, $k_0$}
We call {\it parameter inversions}\/ the transformations of $H_{u_1,k_1,k_0,u_0}$
that preserve the generators $V_1,\check{V_1},\check{V_0},V_0$ and only act on one
of the parameters 
$k_1,u_1,u_0,k_0$ by inversion and change of sign.
We list them in Table \ref{tb:okam},
where we also give the action of these transformations on $\FSH_{a,b,c,d}$.

 \begin{table}[h]
\begin{center} 
\begin{tabular}{||c||c|c|c|c||c|c|c|c||c|c|c||}
 \hline
  & $k_1$ &   $u_1$  & $u_0$  & $k_0$  & $a$ & $b$ & $c$ & $d$  & $T_1$ & $T_0$ & $Z$ \Tstrut\Bstrut\\
   \hline\hline
  $t_1$ &      $-\frac{1}{k_1}$ &   $u_1$  & $u_0$  & $k_0$  & $\frac{1}{b}$ & $\frac{1}{a}$ & $c$ & $d$  & $abT_1$ & $T_0$ & $Z$ \Tstrut\Bstrut\\
  \hline
  $t_2$ &   $k_1$ &   $-\frac{1}{u_1}$  & $u_0$  & $k_0$  & $b$ & $a$ & $c$ & $d$  & $T_1$ & $T_0$ & $Z$ \Tstrut\Bstrut\\
  \hline
  $t_3$ &  $k_1$ &   $u_1$   & $-\frac{1}{u_0}$  & $k_0$  & $a$ & $b$ &$d$ & $c$  & $T_1$ & $T_0$ & $Z$ \Tstrut\Bstrut\\
  \hline
 $t_4$ &   $k_1$ &   $u_1$  & $u_0$  & $-\frac{1}{k_0}$  & $a$ & $b$ &  $\frac{q}{d}$ & $\frac{q}{c}$ & $T_1$ & $\frac{cd}{q}T_0$ & $Z$ \Tstrut\Bstrut\\
  \hline
\end{tabular}
\vspace{0.2cm}
\end{center}
\caption{Parameter inversions.}
\label{tb:okam}
\end{table}

It is immediately seen that the actions on $(a,b,c,d)$ as given in the table
correspond with the parameter inversions in $(k_1,u_1,u_0,k_0)$.
This correspondence also holds on the DAHA level.
For instance, let us show that $t_4(T_0)=q^{-1}cd T_0$ corresponds with
$t_4(\check{V_0},V_0)=(\check{V_0},V_0)$
(for brevity we write $\phi_{c,d}$ instead of $\phi_{a,b,c,d}$):
\mLP
\begin{tikzcd}
\iup\sqrt{qc^{-1}d^{-1}}\,T_0 \arrow[r, "t_4"]\arrow[d,phantom]
&\iup\sqrt{q^{-1}cd}\,T_0
\arrow[d,"\phi_{q/d,q/c}"]&
\\
\check{V_0}\arrow[u, "\phi_{c,d}^{-1}"]\arrow[r,"t_4"]
&\iup\sqrt{q^{-1}cd}\,\phi_{q/d,q/c}(T_0)&\hskip-1.2cm=
\big(\iup\sqrt{q^{-1}cd}\,\big)\big(-\iup\sqrt{q^{-1}cd}\,\big)
\check{V_0}=\check{V_0},
\end{tikzcd}
\mLP
\begin{tikzcd}
-\iup q^{-1}\sqrt{cd}\,ZT_0^{-1}\arrow[r, "t_4"]\arrow[d,phantom]
&-\iup(cd)^{-1/2}\,ZT_0^{-1}\arrow[d,"\phi_{q/d,q/c}"]
\\
V_0\arrow[u, "\phi_{c,d}^{-1}"]\arrow[r,"t_4"]
&-\iup(cd)^{-1/2}\,\phi_{q/d,q/c}(Z)\phi_{q/d,q/c}(T_0)^{-1}=V_0.
\end{tikzcd}
\mLP\indent
Observe the following properties of the parameter inversions
\begin{itemize}
\item
$t_2$ and $t_3$ leave all generators
$T_0$, $T_1$, $Z$ invariant, and hence also $Y$, $D$, $X$;
\item
$t_4$ leaves $T_1$ and $Z$ invariant.
\end{itemize}
As a consequence, we will see that $t_2$, $t_3$, $t_4$  act on the
spherical subalgebra (see Definition \ref{K214}) as well.
Furthermore we will see, quite remarkably, in Section 6 and 7
that $t_2$ and $t_3$ but not $t_4$ act on symmetric and non-symmetric AW
polynomials, while $t_2$ and $t_4$ but not $t_3$ act on symmetric and
non-symmetric AW functions. As mentioned above, we postpone a discussion of
the action of $t_1$ to [P3] because it doesn't restrict to the spherical
subalgebra.

\begin{remark}\label{rmkWD4}
Note that the parameter inversions in Table \ref{tb:okam} exchange the
positions of $a$ and $b$ or of $c$ and $d$, but never mix the pair $(a,b)$
with the pair $(c,d)$.
One symmetry which does a further mixing of parameters is
$$
t_0\colon(a,b,c,d)\mapsto \left(\frac{q}{d},b,c,\frac{q}{a}\right).
$$
However, the action of this transformation on the generators $T_0,T_1,Z$,
or on $\check{V_1},V_1,V_0,\check{V_0}$ is quite involved
and will be the subject of [P2].
Note also that Assumption \ref{K213} is not preserved by $t_0$, but after
addition of further constraints as in Lemma \ref{K143}, the assumption is
preserved.

Notice the action of a further element
\be
\hat t_0:= (t_2 t_4)t_0(t_2 t_4)^{-1}\colon(a,b,c,d)\mapsto(a,c,b,d).
\label{K161}
\ee
So, in their action on parameter space, $t_0,t_2,t_3,t_4$ generate a
group isomorphic to the semidirect
product $S_4\ltimes (\ZZ/2\ZZ)^3$ (where $(\ZZ/2\ZZ)^3$ is isomorphic to
the group of an even
number of sign changes of $1,2,3,4$). We recognize this group as the 
Weyl group $W(D_4)$, see \cite[\S12.1]{Hu}. In [P2] 
the corresponding actions of $t_0,t_2,t_3,t_4$ on the DAHA will be studied.
In \S\ref{sec:D4Zh} we will meet $W(D_4)$ again, but in a different
realization, acting on the spherical subalgebra.

If we take $t_0,t_1,t_2,t_3$ instead of $t_0,t_2,t_3,t_4$ as generators
then another version of $W(D_4)$ acting on the parameter space is obtained.
When we include both $t_1$ and $t_4$, i.e., consider
$\langle t_0,t_1,t_2,t_3,t_4\rangle$, we obtain an action of the Weyl group of
the affine root lattice $D_4^{(1)}$.
This will also be discussed in [P2] and [P3].
\end{remark}

\subsection{Duality and the automorphism $\si$}\label{suse:duality}
\label{se:duality}
Recall our convention about square roots.
Let $a,b,c,d$ satisfy Assumption \ref{K213}. Then
$q^{-1}abcd\in\CC\backslash (-\iy,0]$. Now
\emph{dual parameters} $\td a,\td b,\td c,\td d$ are well-defined by
\be
\td a=\sqrt{q^{-1}abcd},
\quad\td a\td b=ab,
\quad\td a\td c=ac,
\quad\td a\td d=ad.
\label{79}
\ee
Then we have also:
\be
\begin{split}
&\td b=\sqrt{qabc^{-1}d^{-1}},\quad
\td c=\sqrt{qab^{-1}cd^{-1}},\quad
\td d=\sqrt{qab^{-1}c^{-1}d},\\
&\td b\td c=qad^{-1},\quad
\td b\td d=qac^{-1},\quad
\td c\td d=qab^{-1},\quad
\td a\td b^{-1}=q^{-1}cd,\quad
\td a\td c^{-1}=q^{-1}bd,\\
&\td a\td d^{-1}=q^{-1}bc,\quad
\td b\td c^{-1}=bc^{-1},\quad
\td b\td d^{-1}=bd^{-1},\quad
\td c\td d^{-1}=cd^{-1}.
\end{split}
\label{K105}
\ee
Note that $t_4$ acting on $a,b,c,d$ corresponds to $t_2$ acting on
$\td a,\td b, \td c, \td d$, while exchange of $b,c$ corresponds to exchange
of $\td b,\td c$.
Also the following lemma is clear from \eqref{K105}.

\begin{lm}
\label{K143}
Let Assumption \ref{K213} hold.
Then \textup{(i)} $q^{-1}abcd\in\CC\backslash (-\iy,0]$\textup{; (ii)}
the dual parameters are well-defined by \eqref{79}\textup{; (iii)}
$\td a$, $\td b$, $\td c$, $\td d$, $\td a\td b$, $\td a\td b^{-1}$,
$\td c\td d$, $\td c\td d^{-1}$ all have real part $>0$ and
$q^{-1}\td a\td b\td c\td d\in\CC\backslash (-\iy,0]$\textup{; (iv)}
taking duals of dual parameters once more, we recover $a,b,c,d$.

If moreover
$ac$, $ad$, $bc$, $bd$, $ac^{-1}$, $ad^{-1}$, $bc^{-1}$, $bd^{-1}$ have real part
$>0$ then
$\td a\td c$, $\td a\td d$, $\td b\td c$, $\td b\td d$, $\td a\td c^{-1}$,
$\td a\td d^{-1}$, $\td b\td c^{-1}$, $\td b\td d^{-1}$ have real part $>0$.
\end{lm}

By \eqref{eq:SO1} and \eqref{eq:SO1inv} the duality map
$(a,b,c,d)\to(\td a,\td b,\td c,\td d)$ translates to
$(k_1,u_1,u_0,k_0)\to(k_1,k_0,u_0,u_1)$, exchanging $u_1$ and $k_0$.
A corresponding DAHA automorphism
$\si$ which maps $H_{k_1,u_1,u_0,k_0}$ to
$H_{k_1,k_0,u_0,u_1}$ is given in
\cite[Prop.~8.5(ii)]{N-St}, \cite[Cor.~2.7]{St2003} while working in the
basic representation, and in \cite[(2.10)]{Ob} on the DAHA:
\be\label{eq:sig-cal}
\si\colon\begin{cases}
\big(V_1,\check{V_1},\check{V_0},V_0 \big)\mapsto
\big(V_1,V_0,V_0\check V_0 V_0^{-1},V_1^{-1} \check V_1 V_1\big)\colon
H_{k_1,u_1,u_0,k_0}\to H_{k_1,k_0,u_0,u_1},&\\
\left(T_1,T_0,Z\right)\qquad\!\mapsto
\big(T_1,\td a\,T_1^{-1} Z^{-1},a\,T_1^{-1} T_0^{-1}\big)\colon\quad\!\!
\FSH_{a,b,c,d}\to\FSH_{\td a,\td b,\td c,\td d}\,.&
\end{cases}
\ee

Let us show, for instance, that the action of $\si$ on $\FSH_{a,b,c,d}$
implies that $\si(V_0)=V_1^{-1}\check V_1 V_1$ and
$\si(\check V_1)=V_0$\,:
\mLP
\begin{tikzcd}
\iup\sqrt{qc^{-1}d^{-1}}\,T_0\arrow[r, "\si"]\arrow[d,phantom]
&\iup\sqrt{ab}\,T_1^{-1}Z^{-1}\arrow[d,"\phi_{\td a,\td b,\td c,\td d}"]
\\
V_0\arrow[u, "\phi_{a,b,c,d}^{-1}"]\arrow[r,"\si"]
&\iup\sqrt{ab}\,
\phi_{\td a,\td b,\td c,\td d}(T_1)^{-1}
\phi_{\td a,\td b,\td c,\td d}(Z)^{-1}
=V_1^{-1}\check V_1 V_1,
\end{tikzcd}
\mLP
\begin{tikzcd}
\iup\sqrt{ab}\,Z^{-1}T_1^{-1}\arrow[r, "\si"]\arrow[d,phantom]
&\iup\sqrt{a^{-1}b}\,T_0\arrow[d,"\phi_{\td a,\td b,\td c,\td d}"]
\\
\check V_1\arrow[u, "\phi_{a,b,c,d}^{-1}"]\arrow[r,"\si"]
&\iup\sqrt{a^{-1}b}\,\phi_{\td a,\td b,\td c,\td d}(T_0)=V_0
\end{tikzcd}
\mLP\indent
Iteration of \eqref{eq:sig-cal} gives
\be\label{eq:sig2-cal}
\si^2\colon\begin{cases}
\big(V_1,\check V_1,\check V_0,V_0 \big)\mapsto
V_1^{-1}\big(V_1,\check V_1,\check V_0,V_0 \big)V_1\colon
H_{k_1,u_1,u_0,k_0}\to H_{k_1,u_1,u_0,k_0},&\\
\left(T_0,T_1,Z\right)\qquad\!\mapsto T_1^{-1}\left(T_0,T_1,Z\right)T_1\colon
\qquad\!\!\FSH_{a,b,c,d}\to\FSH_{a,b,c,d}.&
\end{cases}
\ee

\begin{remark}
\label{K145}
By \eqref{eq:SO1inv} $u_1u_0=\sqrt{q^{-1}abcd}=\td a$, where $\td a$ is
the dual parameter
given by \eqref{79}. So $\phi(T_1T_0)=\td a\,V_1 V_0$.
So the element $Y$ in \cite[\S2.15 and Theorem 2.22(1)]{N-St}
and the element $Y_1$ in case $n=1$ of \cite[(2.7)]{St2003}
can be identified with our $\td a^{-1}Y$.
By \eqref{eq:sig-cal} we have
\begin{equation}
\si(\td a^{-1}Y)=\td a^{-1}\si(T_1T_0)=Z^{-1},
\end{equation}
compare with \cite[Cor.~2.7]{St2003}.
From \eqref{eq:sig-cal} we also have
$\si(aT_0Z^{-1})=\td a\,T_0Z^{-1}$.
\end{remark}
\subsection{The flip $(a,b,c,d)\mapsto(a,b,c,qd^{-1})$ and the automorphism
$\tau$}
\label{K181}
Let $\tau$ be the DAHA automorphism given on $H_{k_1,u_1,u_0,k_0}$ by
\be\label{eq:tau-new}
\tau\colon\begin{cases}
\big(V_1,\check V_1,\check V_0,V_0 \big) \mapsto
\big(V_1,\check V_1,V_0,V_0\check V_0 V_0^{-1}\big)\colon
H_{k_1,u_1,u_0,k_0}\to H_{k_1,u_1,k_0,u_0},&\\
\qquad\!(T_1,T_0,Z) \mapsto
\big(T_1,q^{-1}c ZT_0^{-1},Z\big)\colon\quad
\FSH_{a,b,c,d}\to\FSH_{a,b,c,qd^{-1}},&
\end{cases}
\ee
so exchanging $u_0$ and $k_0$ while acting on the parameters $k_1,u_1,u_0,k_0$.
Then
\be\label{K172}
\tau^{-1}\colon\begin{cases}
\big(V_1,\check V_1,\check V_0,V_0 \big) \mapsto
\big(V_1,\check V_1,\check V_0^{-1} V_0 \check V_0,\check V_0\big)\colon
H_{k_1,u_1,u_0,k_0}\to H_{k_1,u_1,k_0,u_0},&\\
\qquad\!(T_1,T_0,Z) \mapsto
\big(T_1,q^{-1}c T_0^{-1}Z,Z\big)\colon\quad\,
\FSH_{a,b,c,d}\to\FSH_{a,b,c,qd^{-1}},&\end{cases}
\ee
\be\label{K183}
\tau^2\colon\begin{cases}
\big(V_1,\check V_1,\check V_0,V_0 \big) \mapsto
\big(V_1,\check V_1,V_0 \check V_0 V_0^{-1},
V_0\check V_0 V_0 \check V_0^{-1} V_0^{-1}\big)\colon
H_{k_1,u_1,u_0,k_0}\to H_{k_1,u_1,u_0,k_0},&\\
\qquad\!(T_1,T_0,Z) \mapsto
\big(T_1,ZT_0Z^{-1},Z\big)\colon\qquad\qquad\qquad\qquad\!\!
\FSH_{a,b,c,d}\to\FSH_{a,b,c,d}\,.&\end{cases}
\ee
In combination with the data for $t_3$ and $t_4$ in Table \ref{tb:okam} we see
that
\begin{align}
t_4=\tau t_3 \tau^{-1}\colon& T_0\mapsto q^{-1}c T_0^{-1}Z
\mapsto q^{-1}c T_0^{-1}Z  \mapsto q^{-1}cd T_0\nn\\
\colon& \FSH_{a,b,c,d} \to \FSH_{a,b,c,qd^{-1}}\to
\FSH_{a,b,qd^{-1},c}\to\FSH_{a,b,qd^{-1},qc^{-1}}.
\label{K176}
\end{align}
From \eqref{eq:sig-cal}, \eqref{eq:tau-new} we obtain
\be\label{eq:sigtau3-new}
(\si\tau)^3\colon\begin{cases}
\big(V_1,\check V_1,\check V_0,V_0 \big) \mapsto
V_1^{-2}\big(V_1,\check{V_1},\check{V_0},V_0 \big)V_1^2\colon
H_{k_1,u_1,u_0,k_0}\to H_{k_1,u_1,u_0,k_0},&\\
\qquad\!(T_1,T_0,Z) \mapsto
T_1^{-2}\left(T_1,T_0,Z\right)T_1^2\colon\qquad\!
\FSH_{a,b,c,d}\to\FSH_{a,b,c,d}\,.&\end{cases}
\ee

\begin{remark}
The automorphisms $\si$ and $\tau$ are examples of DAHA automorphisms
which act
on the parameters $k_1,u_1,u_0,k_0$ as permutations. In particular, $\tau$ acts
on these parameters as an adjacent transposition. By cyclic permutation of
$k_1,u_1,u_0,k_0$ and $V_1,\check V_1,\check V_0,V_0$ we can obtain from $\tau$
two other automorphisms which act on the parameters as adjacent transpositions,
and these three adjacent transpositions generate the permutation group
$S_4$. However, when also acting on the
DAHA generators, they generate an infinite group.
The automorphisms $\si$ and $\tau$ generate a subgroup of the latter group
which leaves $T_1$ fixed; see also the next subsection.
\end{remark}

\subsection{Extended modular group action on the DAHA}\label{se:Gamma}
Clearly $\si$ and $\tau$ in their action on 
$k_1,u_1,u_0,k_0$ parameter space generate the group $S_3$ of permutations
of $u_1,u_0,k_0$. 
For the group generated by $\si$ and $\tau$ acting on the DAHA
we have from
(3.5) and (3.11) relations $\sigma^2=\rho$ and $(\sigma\tau)^3=\rho^2$,
where $\rho$ is the conjugation given by
$\rho(U):=V_1^{-1}U V_1$ ($U\in H_{k_1,u_1,u_0,k_0}$).
So modulo conjugation by $T_1$ we have the modular group
$\mathbb P SL_2(\mathbb Z)$, presented as
$\langle \td\si,\td\tau\mid \td\si^2=(\td\si\td\tau)^3=1\rangle$.
Here we follow
Stokman \cite[\S2.6]{St2003} (who worked in the basic
representation) and Oblomkov \cite[Prop.~2.4]{Ob}. However note that the
second relation in \cite[(2.13)]{Ob} is not correct.

If we also consider a further automorphism
\be\label{eq:eta-new}
\eta\colon\begin{cases}
\big(V_1,\check V_1,\check V_0,V_0 \big) \mapsto
\left(-V_1^{-1} , -V_1^{-1} \check V_1^{-1}V_1,
-V_0\check V_0^{-1}V_0^{-1},-V_0^{-1}\right)\colon&\\
\qquad\qquad\qquad\qquad\qquad\qquad\qquad\quad\;\;
H_{k_1,u_1,u_0,k_0;q}\to H_{-u_1^{-1},-k_1^{-1},-k_0^{-1},-u_0^{-1};q^{-1}},&\\
\qquad\!(T_1,T_0,Z) \mapsto
\big(T_1^{-1},T_0^{-1},Z^{-1}\big)\colon
\FSH_{a,b,c,d;q}\to\FSH_{a^{-1},b^{-1},c^{-1},d^{-1};q^{-1}},&
\end{cases}
\ee
(note that $\eta$ does not leave $q$ invariant) then we observe further relations $\eta^2=(\si\eta)^2=(\tau\eta)^2=1$. So the group generated by $\sigma$, $\tau$ and $\eta$ acting on the DAHA is modulo conjugation by $T_1$
the \emph{extended modular group}
$\overline{\Gamma}:=\mathbb P GL_2(\mathbb Z)$ with presentation
$$
\langle \td\si,\td\tau,\td\eta\mid
\td\si^2=(\td\si\td\tau)^3=\td\eta^2=(\td\si\td\eta)^2=(\td\tau\\tdeta)^2=1
\rangle.
$$
Here we follow Oblomkov \cite[Prop.~2.4]{Ob}.

Note that the action of $\overline{\Gamma}$ on parameter space is simply
isomorphic with $S_3\times\ZZ_2$.
Also note that, by \eqref{K83} and \eqref{eq:eta-new},
\be
\eta(Y)=T_1^{-1}Y^{-1}T_1,\quad
\eta(D)=q(abcd)^{-1} D,\quad
\eta(X)=X.
\label{K123}
\ee

\subsection{Spherical braid group}
In \cite{Ma2018} it was shown that the DAHA of type $\check{C_1}C_1$ appears
naturally as quantisation of the group algebra of the monodromy group
associated to  the auxiliary linear system for
the sixth Painlev\'e equation. In \cite{Du-Ma} it was shown that the
analytic continuation of the solutions of the sixth Painlev\'e equation is
described by a natural action of the spherical braid group
$B_3=\langle \beta_1^{\pm 1},\beta_2^{\pm 1}\mid
\beta_1\beta_2\beta_1=\beta_2\beta_1\beta_2, (\beta_1\beta_2)^3=1 \rangle$.
By quantisation, assigning $ M_0  \mapsto V_0,\quad 
M_t  \mapsto  V_0\check V_0 V_0^{-1},\quad 
M_1  \mapsto V_1^{-1}\check V_1 V_1$, this action can be brought to the DAHA
as $\beta_1=\tau$ (see \eqref{eq:tau-new}) and
\begin{equation*}
\beta_2\colon\begin{cases}
\big(V_1,\check V_1,\check V_0,V_0 \big) \mapsto
\left(V_1,\check V_0,\check V_0 \check V_1 \check V_0^{-1},V_0\right)\colon
H_{k_1,u_1,u_0,k_0}\to H_{k_1,-u_0^{-1},-u_1^{-1},k_0},&\\
\qquad\!(T_1,T_0,Z) \mapsto
\big(T_1,T_0,-q\sqrt{abc^{-1}d^{-1}}\,T_1^{-1} Z^{-1} T_0\big)\colon\\
\qquad\qquad\qquad\qquad\qquad\qquad\qquad\qquad\qquad
\FSH_{a,b,c,d}\to
\FSH_{-q^{1/2}\td a/c,-q^{1/2}\td a/d,-q^{1/2}\td a/a,-q^{1/2}\td a/b}.&
\end{cases}
\end{equation*}

It is a straightforward computation to check the braid relation
$\beta_1\beta_2\beta_1=\beta_2\beta_1\beta_2$ and that, up to a global conjugation
by $\check{V_1}$, the Coxeter element $(\beta_1\beta_2)^3$ acts like the identity. 
The inversion $\eta$ given by \eqref{eq:eta-new} was studied also in \cite{L-T}
as an extension of the braid group.

Note that in the formulae for $\beta_1$ and $\beta_2$ the element $V_1$ is always
preserved, and $\eta$ keeps its position (similarly both $\si$ and $\tau$ preserve
$V_1$). This is because in fact the action of $B_3$ and $\overline\Gamma$ is
natural on three elements. 

We can express the modular group generators in terms of braids as
$\si= \beta_1 \beta_2 \beta_1$, $\tau=\beta_1$.
\section{The basic DAHA representation and DAHA automorphisms acting on it}
\label{K187}
\subsection{The basic DAHA representation}

The \emph{basic} (or polynomial) \emph{representation}
$\pi_{a,b,c,d}$  \cite{Sa1999}, \cite{N-St},
\cite[(3.10)--(3.12)]{Koo2007}, \cite[\S5.3]{Koo-Ma} of
$\FSH_{a,b,c,d}$ on the space $\FSA:=\CC[z^{\pm1}]$ of Laurent
polynomials sends the 
generators $Z,T_1,T_0$ to operators $Z$,
$T_1=T_1(a,b)$, $T_0=T_0(c,d)$ acting on $\FSA$ as
\begin{align}
(Zf)(z)&:=z\,f(z),
\label{14}\\
(T_1f)(z)&:=\frac{(a+b)z-(1+ab)}{1-z^2}\,f(z)+
\frac{(1-az)(1-bz)}{1-z^2}\,f(z^{-1})\nn\\
&\;=-ab f(z)+z^{-1}(1-az)(1-bz)\,\frac{f(z)-f(z^{-1})}{z-z^{-1}}\,,
\label{15}\\
(T_0f)(z)&:=\frac{(q^{-1}cd+1)z^2-(c+d)z}{q-z^2}\,f(z)
-\frac{(c-z)(d-z)}{q-z^2}\,f(qz^{-1})\nn\\
&\;=-q^{-1}cdf(z)-z^{-1}(c-z)(d-z)\,\frac{f(z)-f(qz^{-1})}{z-qz^{-1}}\,.
\label{16}
\end{align}
This representation is faithful.
Clearly the expressions \eqref{14}--\eqref{16} are invariant under $t_2$
(exchange of $a,b$) and $t_3$ (exchange of $c,d$). We will discuss their
behaviour under $t_4$ in the next subsection.
The basic representation extends to the space $\FSO$ of analytic functions
on $\CC^*$.

Further function spaces on which the basic representation is acting can be
obtained as cyclic DAHA modules generated by a Gaussian or inverse Gaussian.
In this context we define the \emph{Gaussian} \cite[Def.~2.9]{St2003} by
\be
G_e(z):=\frac1{(ez,ez^{-1};q)_\iy}\,,\quad e,z\in\CC^*,\;
z^{\pm1}\notin eq^{\ZZ_{\ge0}}.
\label{K164}
\ee
From \eqref{15} we see that $T_1$ commutes with $h(Z)$ for any
function $h$ if $h(z)=h(z^{-1})$ (compare also with Corollary \ref{K177}).
So $T_1$ commutes with $G_e(Z)$.

Following \cite[Prop.~2.11]{St2003} we now compute the conjugation of
$T_0$ by suitable Gaussians.

\begin{lm}\label{K165}
We have
\begin{align}
G_d(Z)T_0(c,d) G_d(Z)^{-1}&=
q^{-1}c\,Z T_0(c,qd^{-1})^{-1},\label{K166}\\
G_{qd^{-1}}(Z)^{-1}T_0(c,d) G_{qd^{-1}}(Z)&=
q^{-1}c\,T_0(c,qd^{-1})^{-1}Z.\label{K167}
\end{align}
\end{lm}

\begin{proof}
Note that
\[
\frac{G_d(z)}{G_d(qz^{-1})}=\frac{1-q^{-1}dz}{1-dz^{-1}}\quad\mbox{and}\quad
\frac{G_{qd^{-1}}(qz^{-1})}{G_{qd^{-1}}(z)}=\frac q{z^2}\,
\frac{1-q^{-1}dz}{1-dz^{-1}}.
\]
Then use \eqref{16} and \eqref{K168}.
\end{proof}

Since $G_d(Z)$ and $G_{qd^{-1}}(Z)$ commute
with $Z$ and $T_1$, we conclude:

\begin{cor}\label{K170}
If the basic representation acts on a module $\FSM$ consisting of functions on
$\CC^*$
(like $\FSA$ or $\FSO$) then it acts also on
$\FSM G_d^{-1}$ and $\FSM G_{qd^{-1}}$.
\end{cor}

\begin{remark}\label{K171}
From \eqref{eq:tau-new}, \eqref{K172}, \eqref{K166} and \eqref{K167}
we see how the automorphisms $\tau$ and $\tau^{-1}$
descend to the basic representation:
\begin{align}
G_d(Z) \pi_{a,b,c,d}(U) G_d(Z)^{-1}&=\pi_{a,b,c,qd^{-1}}(\tau(U)),
\label{K169}\\
G_{qd^{-1}}(Z)^{-1} \pi_{a,b,c,d}(U) G_{qd^{-1}}(Z)&=
\pi_{a,b,c,qd^{-1}}(\tau^{-1}(U)),\label{K173}
\end{align}
where $U\in\FSH_{a,b,c,d}$ in both cases.
Formula \eqref{K169} is the case $n=1$ of \cite[Prop.~2.11]{St2003}.
\end{remark}

\subsection{Action of $t_4$ on the basic representation}
\label{K131}
\begin{lm}\label{K138}
Given an element  $U\in\FSH_{a,b,c,d}$, we have
\be
\frac{G_c(Z)}{G_{qd^{-1}}\,(Z)}\,\pi_{a,b,c,d}(U)\,
\frac{G_{qd^{-1}}(Z)}{G_c(Z)}=
\pi_{a,b,qd^{-1},qc^{-1}}(t_4(U)).
\label{K179}
\ee
In particular,
\be
\frac{G_c(Z)}{G_{qd^{-1}}(Z)}\,T_0(c,d)\,
\frac{G_{qd^{-1}}(Z)}{G_c(Z)}=
q^{-1}cd\,T_0(qd^{-1}, qc^{-1}).
\label{K127}
\ee
Hence also
\begin{align}
\frac{G_c(Z)}{G_{qd^{-1}}(Z)}\,Y(a,b,c,d)\,
\frac{G_{qd^{-1}}(Z)}{G_c(Z)}&=
q^{-1}cd\,Y(a,b,qd^{-1}, qc^{-1}),\label{K5}\\
\frac{G_c(Z)}{G_{qd^{-1}}(Z)}\,D(a,b,c,d)\,
\frac{G_{qd^{-1}}(Z)}{G_c(Z)}&=
q^{-1}cd\,D(a,b,qd^{-1}, qc^{-1}).\label{K6}
\end{align}
\end{lm}
\begin{proof}
For \eqref{K179}
apply \eqref{K176} together with \eqref{K169} and \eqref{K173}.
Alternatively, \eqref{K127} can be proved by direct computation using \eqref{16}.

For \eqref{K5} and \eqref{K6} use \eqref{K83}. In particular, for \eqref{K6}
note that by \eqref{K83} and \eqref{K5} both sides of \eqref{K6}
are equal to $q^{-1}cd\big(Y(a,b,qd^{-1}, qc^{-1})+
qabc^{-1}d^{-1}Y(a,b,qd^{-1}, qc^{-1})^{-1}\big)$.
\end{proof}

\begin{remark}
The lower horizontal arrow in the commutative diagram \eqref{K216}
is realized in \eqref{K169}, \eqref{K173}, \eqref{K179} as
a simple conjugation by a suitable function of $Z$. This conjugation
acts on the space of linear operators on $\FSA$.
\end{remark}

As a corollary of \eqref{K5} we obtain:

\begin{prop}\label{K141}
Let $\phi\in\FSO G_{qd^{-1}}$ be an eigenfunction of $Y(a,b,c,d)$
with eigenvalue $\la$, then $\frac{G_c}{G_{qd^{-1}}}\phi\in\FSO G_c$ is an
eigenfunction of $Y(a,b,qd^{-1},qc^{-1})$ with eigenvalue
$q(cd)^{-1}\la$.
\end{prop}

If we have, for $(a,b,c,d,\la)$ in a certain region of $\CC^5$, a family of
eigenfunctions $\phi_\la^{(a,b,c,d)}$ of $Y(a,b,c,d)$ with
$\phi_\la^{(a,b,c,d)}$
analytic in $(a,b,c,d,\la)$ for generic values of $(a,b,c,d)$
then it is not yet guaranteed by Proposition \ref{K141} that
\[
\frac{G_c}{G_{qd^{-1}}}\,\phi_\la^{(a,b,c,d)}=\const
\phi_{q(cd)^{-1}\la}^{(a,b,qd^{-1},qc^{-1})}.
\]
In \eqref{K51} we will have an
instance that such an equality holds.
\subsection{Action of $\si$ on the basic representation}
Different from what we saw for $\tau$ and $t_4$, the action of $\si$ on
the basic representation cannot simply be given by an elementary conjugation.
Following Stokman \cite[(3.1)]{St2003} we see that having such an action of
$\si$ amounts to giving a corresponding Fourier transform, called
\emph{difference Fourier transform}.

\begin{df}
Let $V$ and $W_\si$ be suitable function spaces on $\CC^*$ on
which $\pi_{a,b,c,d}$ and $\pi_{\td a,\td b,\td c, \td d}$, respectively, act.
A linear map $F\colon V\to W_\si$ is called a
\emph{difference Fourier transform} if
\be
F(\pi_{a,b,c,d}(U)f)=\pi_{\td a,\td b,\td c,\td d}(\si(U))(F(f))
\label{K215}
\ee
for all $f\in V$ and all $U\in \FSH_{a,b,c,d}$.
\end{df}

In Stokman's paper this is defined for general rank, but in
\cite[\S8]{St2003} it is worked out in detail in the rank one case.
There $F$ is an integral transform
with integral kernel given by the Cherednik kernel \eqref{K206} or its
symmetric form \eqref{K110}.
\subsection{The spherical subalgebra and the Zhedanov algebra}
In this subsection let $\FSM=\FSA G$ or $\FSO G$ with $G=1$ or $G_{qd^{-1}}$
or $G_d^{-1}$.
We call $f\in\FSM$ \emph{symmetric} if $f(z)=f(z^{-1})$ and we notate
$\FSM^+$ for the subspace of symmetric functions in $\FSM$.
\begin{lm}\label{K174}
For any $g\in \FSM$
\begin{align}
&(T_1 g)(z) = -a b g(z) \iff   g\in\FSM^+\label{eq:T1ab}\\
&(T_1 g)(z) = -g (z) \iff    g(z) =\frac{(1-az)(1-bz)}{z}\,h(z)
\hbox{ with } h\in\FSM^+.\label{eq:T11}
\end{align}
\end{lm}

\begin{proof}
Observe that,
for $g\in  \FSM$,
\begin{equation*}
((T_1 + ab)g)(z) = \frac{ (1-az)(1-bz)}{1-z^2}\,(g(z^{-1})-g(z)).
\end{equation*}
This immediately implies \eqref{eq:T1ab}.
We can also write
\begin{equation}
((T_1 + 1)g)(z)=
\frac{(1-az)(1-bz)}{1-z^2}\,g(z^{-1}) -\frac{(a-z)(b-z)}{1-z^2}\,g(z).
\label{eq:extT0}
\end{equation}
If $g(z)=z^{-1}(1-az)(1-bz) h(z)$ with $h\in\FSM^+$ then the
right-hand side of \eqref{eq:extT0} is~$0$. Vice-versa, for $a\neq a^{-1}$ and
$b\neq b^{-1}$, if the right-hand side of \eqref{eq:extT0} is identically $0$,
then $g(a^{-1})=0$ and $g(b^{-1})=0$. Hence we
can write $g(z)=z^{-1}(1-az)(1-bz) h(z)$ with $h\in\FSM$. Substituting
this in 
$$
\frac{(1-az)(1-bz)}{1-z^2} g(z^{-1}) 
 - \frac{(a-z)(b-z)}{1-z^2} g(z)\equiv 0
$$
yields that
$h(z)=h(z^{-1})$, which proves \eqref{eq:T11}.
\end{proof}

 \begin{cor}\label{K175}
 Given any $g\in \FSM$, there exist two elements $g_1,g_2\in  \FSM^+$
 such that 
\be\label{eq:m-dec}
 g(z)= g_1(z) - \frac{ (1-az)(1-bz)}  {z} g_2(z).
\ee
\end{cor}

\begin{proof}
 We define 
\begin{align}
 g_1 &= \frac{1}{1-a b}\left(T_1 g+ g\right),\label{eq:mg1}\\
 g_2 &= \frac{z}{(1-a z)(1-b z)}\frac{1}{1-a b}\left(T_1 g+a b \, g\right),
 \label{eq:mg2}
 \end{align}
 which gives \eqref{eq:m-dec}. To prove that $g_1$ and $g_2$ are symmetric, it is enough to show that $T_1 g_1=- a b\, g_1$, and 
 $T_1 h_2 = - h_2$
 with $h_2(z)=\frac{(1-a z)(1-b z)}{z}\,g_2(z)$;
 both relations follow immediately from~\eqref{sahi1}.
 \end{proof}

Let $e:=(1-ab)^{-1}(T_1+1)$. By \eqref{sahi1} we have $e^2=e$.

\begin{df} \label{K214}
The \emph{spherical subalgebra} of $\FSH_{a,b,c,d}$ is the image
$\mathcal{SH}_{a,b,c,d}$ of the projection operator $U\mapsto eUe$.
\end{df}
\begin{prop}\quad\\
\textup{\textbf{(a)}}
$\mathcal{SH}_{a,b,c,d}$ is generated as an algebra
by its elements $De$ and $Xe$.\\
\textup{\textbf{(b)}}
In the basic representation on $\FSM$ the subspace
$\FSM^+$ of symmetric functions is
invariant under the action of $\mathcal{SH}_{a,b,c,d}$. The corresponding
representation is faithful.\\
\textup{\textbf{(c)}}
For $f\in\FSM^+$ we have
\be
\pi_{a,b,c,d}(De)f=D(a,b,c,d) f=L^{a,b,c,d}f,
\label{K146}
\ee
where $L=L^{a,b,c,d}$ is the \emph{Askey--Wilson operator
\begin{align}
(Lf)(z)&=L_z\,f(z):=
\big(1+q^{-1}abcd\big)f(z)+\frac{(1-az)(1-bz)(1-cz)(1-dz)}{(1-z^2)(1-qz^2)}\,
\bigl(f(qz)-f(z)\bigr)\nn\\*
&\qquad+\frac{(a-z)(b-z)(c-z)(d-z)}{(1-z^2)(q-z^2)}\,
\bigl(f(q^{-1}z)-f(z)\bigr),\qquad f(z)=f(z^{-1}).\label{K72}
\end{align}}
Also
\[
(\pi_{a,b,c,d}(Xe)f)(z)=(X f)(z)=(z+z^{-1})f(z).
\]
\textup{\textbf{(d)}}
The algebra generated by the operators $L^{a,b,c,d}$ and $X$ acting
on $\FSM^+$ is isomorphic with
the quotient $AW(3,Q_0)$ of the Zhedanov algebra
$AW(3)$ (as defined in \cite[\S2]{Koo2007}) with the ideal generated by
$Q-Q_0$ (see \cite[(2,3), (2,8)]{Koo2007}; here $Q$ is the Casimir element).
The generators $K_0$ and $K_1$ of $AW(3,Q_0)$ correspond under this isomorphism
with $L^{a,b,c,d}$ and $X$, respectively.
\label{K162}
\end{prop}

For the proof of parts (a), (b) see \cite[\S3]{Koo2008}, for part (c), see \cite[Proposition 5.8]{N-St}.
For the proof of part (d) see \cite[\S2]{Koo2007}. See also \cite{T}.

\subsection{Action of $W(D_4)$ on the spherical subalgebra
$\mathcal{SH}_{a,b,c,d}$}\label{sec:D4Zh}

In Remark \ref{rmkWD4}, we observed that $\hat t_0,t_2,t_3,t_4$, acting on
$(a,b,c,d)$ parameter space, generate $W(D_4)$. The action of these
automorphisms on $D$ and $T_1$ in the basic representation induces
an action on $D$. However, while $t_2,t_3,t_4$ preserve $T_1$ and therefore
automatically 
restrict to the spherical subalgebra, $\hat t_0$ doesn't preserve $T_1$ and
therefore its restriction to the spherical 
subalgebra is more complicated and it will be explored in [P2].

From the expression for $L$ in \eqref{K72} we see that $L$ is
invariant under permutations of $a,b,c,d$ while acting on the space
$\FSA^+$ of symmetric Laurent polynomials and that it is only invariant
under permutations of $a,b,c$ while acting on $\FSO^+ G_{qd^{-1}}$. Then $t_4$
also induces an action on $L$. Indeed, by Lemma \ref{K138} and \eqref{K146}
we have
\be
\frac{G_c(Z)}{G_{qd^{-1}}\,(Z)}\,L^{a,b,c,d}\,
\frac{G_{qd^{-1}}(Z)}{G_c(Z)}=
q^{-1}cd\,L^{a,b,qd^{-1}, qc^{-1}}.
\label{K0}
\ee

\begin{remark}
On $L$ considered as a formal operator the permutations of $a,b,c,d$ together
with $t_4$ generate a symmetry group which is isomorphic with
$W(D_4)$, see  \cite[\S2.3]{Koo2023}. This symmetry group can also be
identified with the $W(D_4)$ action discussed in \cite[\S4]{Cr-F-Gab-V},
as we will show in our upcoming paper [P2].
\end{remark}

As a corollary to \eqref{K0} we obtain:

\begin{prop}\label{K75}
Let $\phi\in\FSO^+ G_{qd^{-1}}$ be an eigenfunction of $L^{a,b,c,d}$
with eigenvalue $\la$, then  $\frac{G_c}{G_{qd^{-1}}}\,\phi\in\FSO^+ G_c$  is an
eigenfunction of $L^{a,b,qd^{-1},qc^{-1}}$ with eigenvalue
$q(cd)^{-1}\la$.
\end{prop}

This Proposition is equivalent to Suslov's observation in
\cite[\S6.3, item 2]{Su1989}, already met just after \eqref{K80}.

\begin{remark}\label{K186}
From Proposition \ref{K141} it follows that 
if for $(a,b,c,d,\la)$ in a certain region of $\CC^5$, one has a family of
eigenfunctions $\phi_\la^{(a,b,c,d)}$ of $L^{a,b,c,d}$ with
$\phi_\la^{(a,b,c,d)}$
analytic in $(a,b,c,d,\la)$ for generic values of $(a,b,c,d)$
then $\frac{G_c}{G_{qd^{-1}}}\,\phi_\la^{(a,b,c,d)}$ defines a family of 
eigenfunctions of $L^{(a,b,qd^{-1},qc^{-1})}$ . However, this doesn't guarantee that 
we may have a relation of the following form
\[
\frac{G_c}{G_{qd^{-1}}}\,\phi_\la^{(a,b,c,d)}=\const
\phi_{q(cd)^{-1}\la}^{(a,b,qd^{-1},qc^{-1})}.
\]
In \eqref{K25} we will show an
instance that such an equality holds.
\end{remark}

\section{Some special functions}
\subsection{Askey--Wilson polynomials}\label{suse:AWpol}

Denote and renormalize \emph{Askey--Wilson polynomials} (or briefly
\emph{AW polynomials}) as symmetric Laurent polynomials
\be\label{K71}
E_n^+(z)=E_n^+(z;a,b,c,d)
:=\qhyp43{q^{-n},q^{n-1}abcd,az,az^{-1}}{ab,ac,ad}{q,q}
\ee
with normalization
\be
E_n^+(a^{\pm1})=1.
\ee
This is related to the usual notation and normalization \cite[(7.5.2)]{A-W} by
\be\label{K58}
E_n^+(z;a,b,c,d)=
\frac{a^n}{(ab,ac,ad;q)_n}\,p_n\big(\thalf(z+z^{-1});a,b,c,d\,|\,q\big).
\ee
The notation $E_n^+$ comes from \cite[\S10.6]{N-St}. In \cite{Koo-Ma} we
had $R_n$ instead of $E_n^+$.

By \eqref{K58} and \eqref{K71} $p_n$ has the
expansion
\be\label{K74}
p_n\big(\thalf(z+z^{-1});a,b,c,d\,|\,q\big)=a^{-n}
\sum_{k=0}^n (abq^k,acq^k,adq^k;q)_{n-k}\,
\frac{(q^{-n},q^{n-1}abcd,az,az^{-1};q)_k}{(q;q)_k}\,q^k.
\ee
Although not completely visible from \eqref{K74}, $p_n$ is symmetric in $a,b,c,d$
\cite[\S7.5]{A-W}.
The nontrivial symmetry in $a$ and $b$ follows from
Sears' transformation \cite[(III.15)]{A-W}. It
can be written for $E_n^+(z)$ as
\be \label{K77}
E_n^+(z;a,b,c,d)=\left(\frac ab\right)^n\,
\frac{(bc,bd;q)_n}{(ac,ad;q)_n}\,E_n^+(z;b,a,c,d).
\ee

We will also need the (in $z$) \emph{monic} AW polynomial
\be
P_n^+(z)=P_n^+(z;a,b,c,d):=\frac{(ab,ac,ad;q)_n}{(q^{n-1}abcd;q)_n\,a^n}\,
E_n^+(z;a,b,c,d)
=\frac{p_n\big(\thalf(z+z^{-1});a,b,c,d\,|\,q\big)}{(q^{n-1}abcd;q)_n}\,,
\label{K14}
\ee
which is also symmetric in $a,b,c,d$.
The notation $P_n^+$ comes from \cite[\S5.3]{N-St}. In \cite{Koo2007}
we had $P_n$ instead of $P_n^+$.

Under suitable constraints on the parameters the AW polynomials are
orthogonal polynomials with explicit orthogonality relations \cite{A-W}.
They are eigenfunctions of the Askey--Wilson operator (AW operator)
$L=L^{a,b,c,d}$ given  by \eqref{K72}:
\be
(LE_n^+)(z)=(q^{-n}+abcdq^{n-1})E_n^+(z).
\label{K7}
\ee
It is worth remembering that the operator $L^{a,b,c,d}$ is the restriction of the
operator $D(a,b,c,d)$ (in the basic representation)
to the symmetric functions in the representation space.

From \eqref{K71} we see the \emph{duality} for AW polynomials:
\be
E_n^+(a^{-1} q^{-m};a,b,c,d)=E_m^+(\td a^{-1} q^{-n};\td a,\td b,\td c,\td d),\quad
m,n\in\ZZ_{\ge0}.
\label{K137}
\ee

\subsection{Askey--Wilson functions}
\label{K120}
We will use the notation \cite[(2.1.11)]{Gas-R} for \emph{very-well-poised series},
and we will further abbreviate this notation in the case of
\emph{very-well-poised-balanced series} by
\begin{align}
&{}_8W_7(a;b,c,d,e,f)={}_8W_7\left(a;b,c,d,e,f;q,\frac{q^2 a^2}{bcdef}\right)
\nn\\
&:=\qhyp87{a,qa^{1/2},-qa^{1/2},b,c,d,e,f}
{a^{1/2},-a^{1/2},qa/b,qa/c,qa/d,qa/e,qa/f}{q,\frac{q^2 a^2}{bcdef}},
\qquad\left|\frac{q^2 a^2}{bcdef}\right|<1.\label{K2}
\end{align}
Note that this is symmetric in $b,c,d,e,f$.

Koelink and Stokman \cite[(3.2)]{Koe-St} define the \emph{Askey--Wilson function}
(or briefly \emph{AW function}) $\phi_\ga$ as follows.
\begin{align}
&\phi_\ga(z)=\phi_\ga(z;a,b,c,d)=\phi_\ga(z;a,b,c,d;q):=
\frac{(q\td a d^{-1}\ga z,q\td a d^{-1}\ga z^{-1};q)_\iy}
{(bc,q\ga\td a ad^{-1},\ga\td a^{-1}bc,qd^{-1}z,qd^{-1}z^{-1};q)_\iy}
\nn
\sLP
&\qquad\qquad\qquad\times
{}_8W_7(\td a ad^{-1}\ga;az,az^{-1},
\td a\ga,\td b\ga,\td c\ga)\quad
|\ga|>|q\td d^{-1}|\label{K3}.
\end{align}
They use it as a kernel in a Fourier type integral transform, called the
Askey--Wilson function transform, for which they give an inversion formula.

This function was earlier studied by Suslov
\cite[(6.18)]{Su1989}, \cite[(2.7)]{Su1997}, \cite[(2.7)]{Su2001} (notations and
normalizations vary). Another representation of $\phi_\ga(z)$ can be obtained
from \eqref{K3} by applying \cite[(III.23)]{Gas-R}
(there replace $a;b,c,d,e,f$ by
$\td a ad^{-1}\ga;\td b\ga,\td c\ga,az,az^{-1},\td a\ga$):
\begin{align}
&\phi_\ga(z;a,b,c,d)=
\frac{(qa \td d^{-1}\ga^{-1} z^{-1},qa\td d^{-1}\ga z^{-1};q)_\iy}
{(\td b\td c,q\td a a\td d^{-1}z^{-1},a^{-1}\td b\td cz^{-1},q\td d^{-1}\ga,
q\td d^{-1}\ga^{-1};q)_\iy}
\nn
\sLP
&\qquad\qquad\qquad\times
{}_8W_7(\td a a\td d^{-1}z^{-1};\td a\ga,\td a\ga^{-1},
az^{-1},{bz^{-1}},cz^{-1}),\quad
|z|<|q^{-1}d|\label{K78}.
\end{align}
The representation \eqref{K78} of $\phi_\ga$ earlier occurred in papers by
Ismail \& Rahman \cite[(1.12)]{Is-R} and by
Rahman \cite[(1.32), (1.35)]{R1992}, \cite[(1.19)]{R1999} (again with varying
notations and normalizations), see also Haine \& Iliev
\cite[(2.4), (2.8a)]{Ha-Il}. Suslov \cite[(2.7), (2.10)]{Su2001} gave both
representations \eqref{K3} and \eqref{K78}.

The combination of \eqref{K3} and \eqref{K78} states the \emph{duality}
\be
\phi_\ga(z;a,b,c,d)=\phi_{z^{-1}}(\ga^{-1};\td a,\td b,\td c,\td d).
\label{K42}
\ee
By \eqref{K3} and \eqref{K78} $\phi_\ga(z)$ is symmetric in $z$ and $\ga$:
\[
\phi_\ga(z)=\phi_\ga(z^{-1}),\qquad
\phi_\ga(z)=\phi_{\ga^{-1}}(z).
\]

Stokman \cite[(3.1)]{St2002} defines $\phi_\ga$ with a normalization different from
\eqref{K3};
we denote it here by $\phi_\ga^{\textup{S}}$:
\be
\phi_\ga^{\textup{S}}:=\frac{(qabcd^{-1};q)_\iy}{(qbd^{-1},qcd^{-1};q)_\iy}\,\phi_\ga.
\label{K57}
\ee
In \cite[Dfn.~6.12, Thm.~6.20]{St2003} Stokman uses yet another notation
and normalization:
\begin{align}
&\goE^+(\ga;z)=\goE^+(\ga;z;a,b,c,d):=
(bc,qad^{-1},qa^{-1}d^{-1};q)_\iy\,\phi_\ga(z;a,b,c,d)
\label{K53}\\
&=
\frac{(q\ga\td a d^{-1}z,q\ga\td a d^{-1}z^{-1},qad^{-1},qa^{-1}d^{-1};q)_\iy}
{(q\ga\td a ad^{-1},\ga\td a^{-1}bc,qd^{-1}z,qd^{-1}z^{-1};q)_\iy}
\,{}_8W_7(\ga\td a ad^{-1};az,az^{-1},
\ga\td a,\ga ab\td a^{-1},\ga ac\td a^{-1})\label{K54}
\end{align}
From \eqref{K54} we observe the normalization
\be
\goE^+(\ga;a^{\pm1};a,b,c,d)=1.
\label{K90}
\ee
Again we have symmetry in $z$ and $\ga$,
\begin{equation*}
\goE^+(\ga;z)=\goE^+(\ga;z^{-1}),\quad \goE^+(\ga;z)=\goE^+(\ga^{-1};z),
\end{equation*}
and duality
\be
\goE^+(\ga;z;a,b,c,d)=\goE^+(z,\ga;\td a,\td b,\td c,\td d).
\label{61}
\ee
Hence, by \eqref{K90} and \eqref{61},
\be
\goE^+(\td a^{\pm 1};z;a,b,c,d)=1.
\label{K119}
\ee

\begin{remark}
If we substitute $\ga=\td a$ in \eqref{K54}
 then, by \eqref{K119} and \eqref{K2},
 \be
 1=\frac{(abcz,abcz^{-1},qad^{-1},qa^{-1}d^{-1};q)_\iy}
{(a^2bc,bc,qd^{-1}z,qd^{-1}z^{-1};q)_\iy}
\,{}_6W_5(q^{-1}a^2bc;az,az^{-1},q^{-1}abcd;q,qa^{-1}d^{-1}).
\label{K125}
\ee
Formula \eqref{K125} is equivalent with the evaluation formula
\cite[(II.20)]{Gas-R}.
\end{remark}

By \cite[(III.36)]{Gas-R} we can write
\begin{align}
&\goE^+(\ga;z;a,b,c,d)\nn\\
&=\qhyp43{az,az^{-1},\td a\ga,\td a\ga^{-1}}{ab,ac,ad}{q,q}
+\frac{(az,az^{-1},\td a\ga,\td a\ga^{-1},
qbd^{-1},qcd^{-1},qa^{-1}d^{-1};q)_\iy}
{(qd^{-1}z,qd^{-1}z^{-1},q\ga\td a a^{-1}d^{-1},q\ga^{-1}\td a a^{-1}d^{-1},
ab,ac,q^{-1}ad;q)_\iy}\nn\\
&\qquad\qquad\qquad\qquad\qquad\qquad\qquad\qquad\times
\qhyp43
{qd^{-1}z,qd^{-1}z^{-1},q\td d^{-1}\ga,q\td d^{-1}\ga^{-1}}
{qbd^{-1},qcd^{-1},q^2 a^{-1}d^{-1}}{q,q}\label{K55}
\end{align}
Hence, as observed in \cite[\S3]{Koe-St}, $\goE^+(\ga;z;a,b,c,d)$ extends for generic
parameters $a,b,c,d$ to a meromorphic function in $z$ and $\ga$.
More precisely, using the definition \eqref{K164} of a Gaussian, we have:
\begin{prop}\label{K91}
The expression
$\displaystyle
\frac{\goE^+(\ga;z;a,b,c,d)}{G_{qd^{-1}}(z)G_{q\td d^{-1}}(\ga)}$
is an entire complex analytic function in $z+z^{-1}$ and $\ga+\ga^{-1}$.
Equivalently, $\goE^+(\ga;z;a,b,c,d)$ belongs to $\FSO^+G_{qd^{-1}}$ as a
function of $z$ and to $\FSO^+G_{q\td d^{-1}}$ as a function of $\ga$.
\end{prop}

For $\ga=q^n\td a$, $n=0,1,2,\ldots$, we get from \eqref{K55}
an AW polynomial \eqref{K71}:
\be\label{K13}
\goE^+(q^n\td a;z;a,b,c,d)
=\qhyp43{q^{-n},q^{n-1}abcd,az,az^{-1}}{ab,ac,ad}{q,q}
=E_n^+(z;a,b,c,d).
\ee
Hence the eigenvalue equation \eqref{K7} for AW polynomials
can be rewritten as
\be
\td a^{-1}L_z\,\goE^+(q^n\td a;z)=
\big(q^n \td a+(q^n \td a)^{-1}\big)
\goE^+(q^n\td a;z),\qquad n\in\ZZ_{\ge0}\,.
\label{K19}
\ee
As stated in \cite[\S3]{Koe-St},
this extends to an eigenvalue equation for $\goE^+(\ga;z)$:
\be
\td a^{-1}L_z\,\goE^+(\ga;z)=(\ga+\ga^{-1})\,\goE^+(\ga;z).
\label{K18}
\ee
Here $L$ is still given by \eqref{K72}. Note that the first term on the right in
\eqref{K72} is missing in \cite[(2.4)]{Koe-St}, so that it occurs there in the
eigenvalue \cite[(2.5)]{Koe-St}.

\begin{remark}
The eigenvalue equation \eqref{K18} was probably first given by Suslov
\cite[\S6.3]{Su1989}. He derived it by observing that both terms
on the right-hand side of \eqref{K55} are also eigenfunctions of $\td a^{-1}L_z$
with eigenvalue $\ga+\ga^{-1}$. Indeed, put
\be
\goR(\ga;z;a,b,c,d):=\qhyp43{az,az^{-1},\td a\ga,\td a\ga^{-1}}{ab,ac,ad}{q,q},
\label{K81}
\ee
so $\goR(q^\nu\td a;z;a,b,c,d)=E_\nu^+(z;a,b,c,d)$ with $E_\nu^+$ given by
\eqref{K71} with $n$ replaced by $\nu$.
Now \eqref{K55} can be rewritten as
\begin{align}
&\goE^+(\ga;z;a,b,c,d)=\goR(\ga;z;a,b,c,d)\nn\\
&+\frac{(\td a\ga,\td a\ga^{-1},
qbd^{-1},qcd^{-1},qa^{-1}d^{-1};q)_\iy}
{(q\ga\td a a^{-1}d^{-1},q\ga^{-1}\td a a^{-1}d^{-1},
ab,ac,q^{-1}ad;q)_\iy}\,\frac{(az,az^{-1};q)_\iy}{(qd^{-1}z,qd^{-1}z^{-1};q)_\iy}\,
\goR(\ga;z;qd^{-1},b,c,qa^{-1}).
\label{K80}
\end{align}
The eigenvalue equation for $\goR(\ga;z;a,b,c,d)$ is proved in the same way as
in the polynomial case. For the second term on the right in \eqref{K80}
Suslov \cite[\S6.3, item 2]{Su1989} observes that, if $f(z;a,b,c,d)$ is an
eigenfunction of $L_z^{a,b,c,d}$ with eigenvalue $\lambda$ then so is
$\frac{(az,az^{-1};q)_\iy}{(qd^{-1}z,qd^{-1}z^{-1};q)_\iy}\,
f(z;qd^{-1},b,c,qa^{-1})$. We will come back to this symmetry of $L_z$ in
Proposition \ref{K75}.
\end{remark}

In \cite[(4.4), (4.7)]{St2002} Stokman derived a Cherednik type kernel
for the
AW function. This is a formula expanding $\goE^+(\ga;z;a,b,c,d)$
in terms of products of two AW polynomials, one depending on $z$
and one on $\ga$. By \eqref{K57}, \eqref{K53} and \eqref{K14} we can rewrite his
formula as follows.

\begin{prop}\label{K185}
The AW function can be given by the following expansion formula.
\begin{align}
&\goE^+(\ga;z;a,b,c,d)=\frac{(bc,qad^{-1},qbd^{-1},qcd^{-1},qa^{-1}d^{-1};q)_\iy}
{(qabcd^{-1};q)_\iy
(qd^{-1}z,qd^{-1}z^{-1},q\td d^{-1}\ga,q\td d^{-1}\ga^{-1};q)_\iy}\nn\\
&\qquad\times\sum_{m=0}^\iy (-1)^m (ad)^{-m} q^{\half m(m+1)}\,
\frac{1-q^{2m}abcd^{-1}}{1-abcd^{-1}}\,
\frac{(ab,ac,abcd^{-1};q)_m}
{(qbd^{-1},qcd^{-1},q;q)_m}\nn\\
&\qquad\qquad\qquad\qquad\qquad\times E_m^+(z;a,b,c,qd^{-1})\,
E_m^+(\ga;\td a,\td b,\td c,q\td d^{-1})
\label{K110}
\end{align}
The infinite sum converges absolutely and uniformly for
$(\ga,z)$ in compacta of $\CC^*\times\CC^*$.
\end{prop}

As observed in \cite{St2003}, Proposition \ref{K91} and the duality
\eqref{61} immediately follow from the expansion formula.

\begin{remark}
The normalization formula \eqref{K119} in combination with \eqref{K110} and
\eqref{K58} gives rise to the following expansion formula
\cite[Remark 4.8]{St2002} for the inverse Gaussian.
\begin{align}
&\frac1{G_d(z)}=(dz,dz^{-1};q)_\iy
=\frac{(ad,bd,cd;q)_\iy}{(abcd;q)_\iy}\nn\\
&\quad\times
\sum_{m=0}^\iy (-a^{-1}d)^m q^{\half m(m-1)}\,
\frac{1-q^{2m-1}abcd}{1-q^{-1}abcd}\,\frac{(q^{-1}abcd,ab,ac;q)_m}{(bd,cd,q;q)_m}
\,E_m^+(z;a,b,c,d).\label{K184}
\end{align}
By \cite[Remark 4.5]{St2002} this is a special case of a dual connection formula
for AW polynomials
and continuous dual Hahn polynomials which is essentially contained in
\cite[(7.6.8), (7.6.9)]{Gas-R}.
Furthermore, as observed in \cite[Remark 4.8]{St2002}, formula \eqref{K184}
reduces for $(a,b,c,d)=(1,-1,-q^{1/2},q^{1/2})$ to Jacobi's triple product
formula \cite[(1.6.1)]{Gas-R}, since
$E_m^+(z;1,-1,-q^{1/2},q^{1/2})=\thalf(z^m+z^{-m})$ ($m\ne0$) and $=1$ ($m=0$).
\end{remark}

\section{Action of DAHA automorphisms on Askey--Wilson polynomials and functions}
As observed in \S\ref{sec:D4Zh}, the AW operator $L^{a,b,c,d}$ is invariant
under permutations of $a,b,c,d$, while it satisfies the symmetry \eqref{K0} under
the action of $t_4\colon(a,b,c,d)\mapsto(a,b,qd^{-1},qc^{-1})$. These symmetries
generate a group isomorphic to $W(D_4)$. In this section we study the possible
action of these symmetries on the eigenfunctions of $L$. First we consider this
for the AW polynomials, but there no satisfactory action of $t_4$ can be found.
Then we consider it for the AW functions, where $t_4$ acts nicely but $t_3$
(interchanging $c$ and $d$) fails.

\subsection{An attempted action of $t_4$ on Askey--Wilson polynomials}
Let us try to apply Proposition \ref{K75} to AW polynomials.
By \eqref{K7} we see that the eigenfunction $E_n^+(z;a,b,c,d)$ of
$L_z^{a,b,c,d}$ has eigenvalue
$q^{-n}+abcdq^{n-1}$ and that
the eigenfunction $E_m^+(z;a,b,qd^{-1},qc^{-1})$ of
$L_z^{a,b,qd^{-1},qc^{-1}}$ has eigenvalue
$q^{-m}+abc^{-1}d^{-1}q^{m+1}$. If $cd=q^{m-n+1}$ then the second eigenvalue is
equal to $qc^{-1}d^{-1}$ times the first eigenvalue. Conversely, if
$q^{-m}+abc^{-1}d^{-1}q^{m+1}=qc^{-1}d^{-1}(q^{-n}+abcdq^{n-1})$ for all
$m,n\in\ZZ_{\ge0}$ with $n-m$ fixed then $cd=q^{m-n+1}$. This shows,
in view of Proposition \ref{K75}, that $E_n^+(z;a,b,c,d)$ and
$\frac{G_c(z)}{G_{qd^{-1}}(z)}E_m^+(z;a,b,qd^{-1},qc^{-1})$ are
proportional if and only if $cd=q^{m-n+1}$.
Without loss of generality we assume that $n\ge m$. It is also convenient
to work with the normalization and expansion \eqref{K74} because of the
symmetry in $a,b,c,d$. Then we find
\begin{multline} \label{K76}
p_n\big(\thalf(z+z^{-1});a,b,c,d\,|\,q\big)
=(-c)^{-(n-m)} q^{-\half(n-m)(n-m-1)}(q^m ab;q)_{n-m}\\
\times(cz,cz^{-1};q)_{n-m}\,
p_m\big(\thalf(z+z^{-1});a,b,qd^{-1},qc^{-1}\,|\,q\big),\quad
cd=q^{m-n+1}.
\end{multline}

Note that for $cd=q^{m-n+1}$ indeed
$\frac{G_c(z)}{G_{qd^{-1}}(z)}=(cz,cz^{-1};q)_{n-m}$.
\paragraph{Proof of \eqref{K76}.}
From \eqref{K74}, the symmetry in $a,b,c,d$,
the condition $cd=q^{m-n+1}$ and the vanishing for $k<n-m$ of
$(q^{m-n+k+1};q)_{n-k}$ we see that
\begin{align*}
&p_n\big(\thalf(z+z^{-1});c,a,b,d\,|\,q\big)
=c^{-n}\sum_{k=n-m}^n (acq^k,bcq^k,q^{m-n+k+1};q)_{n-k}\,
\frac{(q^{-n},q^m ab,cz,cz^{-1};q)_k}{(q;q)_k}\,q^k\\
&=c^{-n}\sum_{j=0}^m (acq^{n-m+j},bcq^{n-m+j},q^{j+1};q)_{m-j}\,
\frac{(q^{-n},q^m ab,cz,cz^{-1};q)_{n-m+j}}{(q;q)_{n-m+j}}\,q^{n-m+j}\\
&=(-c)^{-(n-m)} q^{-\half(n-m)(n-m-1)}(q^m ab;q)_{n-m}(cz,cz^{-1};q)_{n-m}
\,q^{-m(n-m)}c^{-m}\\
&\qquad\qquad\times \sum_{j=0}^m (acq^{n-m+j},bcq^{n-m+j},q^{n-m+j+1};q)_{m-j}\,
\frac{(q^{-m},q^n ab,q^{n-m}cz,q^{n-m}cz^{-1};q)_j}{(q;q)_j}\,q^j\\
&=(-c)^{-(n-m)} q^{-\half(n-m)(n-m-1)}(q^m ab;q)_{n-m}(cz,cz^{-1};q)_{n-m}
\,(qd^{-1})^{-m}\\
&\qquad\quad\times
\sum_{j=0}^m (ad^{-1}q^{j+1},bd^{-1}q^{j+1},c^{-1}d^{-1}q^{j+2};q)_{m-j}\,
\frac{(q^{-m},q^{m+1}abc^{-1}d^{-1},qd^{-1}z,qd^{-1}z^{-1};q)_j}{(q;q)_j}\,q^j\\
&=(-c)^{-(n-m)} q^{-\half(n-m)(n-m-1)}(q^m ab;q)_{n-m}(cz,cz^{-1};q)_{n-m}\\
&\qquad\qquad\qquad\qquad\qquad\qquad\qquad\qquad\times
p_m\big(\thalf(z+z^{-1});qd^{-1},a,b,qc^{-1}\,|\,q\big).\qed
\end{align*}

As we just saw, if we want to realize the action of $t_4$ on AW
polynomials, then we have to assume that $cd$ is an integer power of $q$.
With a continuous spectrum of $L$, as for AW functions
as eigenfunctions, the condition on $cd$ can be dropped, as we will see
in the next subsection.

\subsection{Action of $t_4$ and $t_2$ on Askey--Wilson functions}
We will work with the AW function $\goE^+(\ga;z;a,b,c,d)$ as
defined by \eqref{K54}, and also given as a symmetric rank one
Cherednik kernel,
see Proposition \ref{K185}.
Note that this function is symmetric in $b$ and $c$, as can be read off
from any of its expressions \eqref{K54}, \eqref{K55} or \eqref{K110}
(respectively using an obvious symmetry of ${}_8W_7$ or ${}_4\phi_3$ or
$E_n^+$).

The following theorem gives the action of $t_4$ and $t_2$ 
on the AW function.

\begin{theorem}\label{th-sym-t4}
We have
\begin{align}
\goE^+(\ga;z;a,b,c,d)&=\frac{(qad^{-1},qa^{-1}d^{-1};q)_\iy}{(ac,a^{-1}c;q)_\iy}\,
\frac{G_{qd^{-1}}(z)}{G_c(z)}\,\goE^+(\ga;z;a,b,qd^{-1},qc^{-1}),\label{K25}\\
\goE^+(\ga;z;a,b,c,d)&
=\frac{(bc,qa^{-1}d^{-1};q)_\iy}{(ac,qb^{-1}d^{-1};q)_\iy}\,
\frac{G_{q\td d^{-1}}(\ga)}{G_{\td c}(\ga)}\,
\goE^+(\ga;z;b,a,c,d).
\label{K56}
\end{align}
\end{theorem}

\paragraph{First Proof}
Formula \eqref{K25} is immediately seen from \eqref{K54}
by the symmetry in $b,c,d,e$ of the ${}_8W_7$ in \eqref{K2}.
Formula \eqref{K56} can be obtained from \cite[(III.23)]{Gas-R} with
the substitutions
\[
a\to \ga\td a a d^{-1},\quad
b\to az,\quad
c\to az^{-1},\quad
d\to\ga \td a^{-1} ac,\quad
e\to\ga\td a,\quad
f\to\ga\td a^{-1} ab,
\]
together with \eqref{K54}.\qed
\paragraph{Second Proof}
Formula \eqref{K110} shows
that the infinite sum (without the factors in front of the sum)
is symmetric in $a,b,c,qd^{-1}$ because
the AW polynomials $R_m^+(z;a,b,c,d)$ are symmetric in $b,c,d$ and
satisfy the symmetry \eqref{K77} in $a$ and $b$.
Then the symmetry in $c,qd^{-1}$ yields \eqref{K25} and the symmetry in $a,b$
yields \eqref{K56}.\qed
\begin{remark}\label{K94}\quad
\begin{enumerate}
\item
Formula \eqref{K25} is in agreement with Remark \ref{K186}.
Indeed, 
$\goE^+(\ga;z;a,b,c,d)$ is an eigenfunction of  $L^{a,b,c,d}$
with eigenvalue $(\ga+\ga^{-1})\sqrt{q^{-1}abcd}$, while
$\goE^+(\ga;z;a,b,qd^{-1},qc^{-1})$
is an eigenfunction of $L_z^{a,b,qd^{-1},qc^{-1}}$ with eigenvalue
$(\ga+\ga^{-1})\sqrt{qabc^{-1}d^{-1}}$.
\item
For $\ga=q^n\td a$ \eqref{K56} combined with \eqref{K13} simplifies to the
well-known \eqref{K77}.
\item
Formulas \eqref{K25} and \eqref{K56} are equivalent by the duality \eqref{61}.
Indeed, use \eqref{61} in the left-hand side of \eqref{K25}, and also in the
right-hand side in the form
\begin{align}
&\goE^+(\ga;z;a,b,qd^{-1},qc^{-1})
=\goE^+\left(z;\ga;\sqrt{\frac{qab}{cd}},ab\sqrt{\frac{cd}{qab}},
qad^{-1}\sqrt{\frac{cd}{qab}},qac^{-1}\sqrt{\frac{cd}{qab}}\;\right)\nn\\
&=\goE^+\left(z;\ga;\frac{ab}{\sqrt{q^{-1}abcd}},\sqrt{q^{-1}abcd},
\frac{ac}{\sqrt{q^{-1}abcd}},\frac{ad}{\sqrt{q^{-1}abcd}}\right)
=\goE^+(z,\ga;\td b,\td a,\td c,\td d).
\label{K147}
\end{align}
Thus \eqref{K56} can be obtained by combining \eqref{K25} with \eqref{61} and
\eqref{K147}.
\item
Altogether we have seen that the AW functions have symmetries
which come from the subgroup of $W(D_4)$ consisting of
permutations of $a,b,c,qd^{-1}$.
\end{enumerate}
\end{remark}

\section{Non-symmetric Askey--Wilson functions and their symmetries}

\subsection{Non-symmetric Askey--Wilson polynomials}
\label{K194}
By \cite[\S3]{N-St}, see also \cite[\S1]{Sa2000} and
\cite[(4.4)--(4.7)]{Koo2007},
\emph{non-symmetric Askey--Wilson polynomials}
$P_n(z;a,b,c,d)=P_n(z)$ ($n\in\ZZ$) are Laurent polynomial eigenfunctions
of $Y$ as follows.
\begin{align}
YP_{-n}&=q^{-n} P_{-n},\qquad\;\; n\in\ZZ_{>0},
\label{K193}\\
YP_n&=q^{n-1}abcd P_n,\quad n\in\ZZ_{\ge0},
\label{K195}
\end{align}
such that
\be
P_{-n}(z)-z^{-n}\in\Span\{z^{-n+1},\ldots,z^{n-1}\},\quad
P_n(z)-z^n\in\Span\{z^{-n},\ldots,z^{n-1}\}.
\ee
This characterizes them uniquely.

Let $\la_n:=q^{-n}+q^{n-1}abcd$ (the eigenvalue of the AW operator
for $P_n^+(z)$).
By \cite[\S\S 5.2, 5.3]{N-St}, see also Proposition 3.1 and (3.16), (3.17),
(3.20), (3.21) in \cite{Koo2007}, there is an alternative characterization
of the (symmetric) AW polynomials $P_n^+(z)$ ($n\in\ZZ_{\ge0}$) and
an analogous definition of the \emph{anti-symmetric Askey--Wilson polynomials}
$P_n^-(z)$ ($n\in\ZZ_{>0}$) by
\begin{equation*}
\begin{split}
&DP_n^+=\la_nP_n^+\;\mbox{and}\;
T_1 P_n^+=-abP_n^+\;\mbox{such that}\;
P_n^+(z)-z^n\in\Span\{z^{-n},\ldots,z^{n-1}\},\\
&DP_n^-=\la_nP_n^-\;\;\mbox{and}\;\;
T_1 P_n^-=-P_n^-\;\;\mbox{such that}\;\;
P_n^-(z)-z^n\in\Span\{z^{-n},\ldots,z^{n-1}\}.
\end{split}
\end{equation*}
The notations $P_{\pm n}$, $P_n^+$ and $P_n^-$ come from
\cite[\S3.5]{N-St}.
In \cite{Sa2000} and \cite{Koo2007} the corresponding notations are
$P_n$, $Q_n$ and $E_n$, respectively.

By \cite[(3.17), (4.3), (4.9), (4.10)]{Koo2007}
$P_n^-(z)$ and $P_{\pm n}$ can be expressed by
\begin{align}
P_n^-(z;a,b,c,d)&=a^{-1}b^{-1}z^{-1}(1-az)(1-bz)P_{n-1}^+(z;qa,qb,c,d),
\label{K99}\\
P_{-n}(z;a,b,c,d)&=\frac{ab}{ab-1}\,\big(P_n^+(z;a,b,c,d)-P_n^-(z;a,b,c,d)\big),
\label{K100}\\
P_{n}(z;a,b,c,d)&=\frac{1}{(1-ab)(1-q^{2n-1}abcd)}\,
\big((1-q^n ab)(1-q^{n-1}abcd)P_n^+(z;a,b,c,d)\nn\\
&\qquad\qquad\qquad\qquad\qquad-ab(1-q^n)(1-q^{n-1}cd)P_n^-(z;a,b,c,d)\big),
\label{K101}\\
P_0(z;a,b,c,d)&=1,\label{K102}
\end{align}
where $n\in\ZZ_{>0}$\,.
From these expressions and from the symmetry of $P_n^+(z)$ in
$a,b,c,d$ we conclude that $P_n^-(z)$ and $P_{\pm n}(z)$ are symmetric in $a,b$
and symmetric in $c,d$. This also follows because $Y$, $D$ and $T_1$
obey these symmetries.

Following \cite[Definition 10.6(i)]{N-St} we define
\emph{renormalized non-symmetric Askey--Wilson polynomials} by
\be
E_n(z;a,b,c,d):=\frac{P_n(z;a,b,c,d)}{P_n(a^{-1};a,b,c,d)}\,,\quad n\in\ZZ.
\ee
Clearly $P_n^-(a^{-1};a,b,c,d)=0$. Hence, by \eqref{K99}, \eqref{K100},
\eqref{K101} and \eqref{K14},
\be
\begin{split}
&P_{-n}(z;a,b,c,d)=\frac{ab}{ab-1}\,
\frac{(ab,ac,ad;q)_n}{(q^{n-1}abcd;q)_n\,a^n}\,
E_{-n}(z;a,b,c,d),\quad n\in\ZZ_{>0}\,,\\
&P_n(z;a,b,c,d)=\frac{(qab,ac,ad;q)_n}{(q^nabcd;q)_n\,a^n}\,
E_n(z;a,b,c,d),\quad n\in\ZZ_{\ge0}\,.
\end{split}
\label{K202}
\ee
We see that $E_{\pm n}(z;a,b,c,d)$ is still symmetric in $c,d$
while for $a,b$ we have the symmetry
\be
E_{\pm n}(z;b,a,c,d)=\frac{b^n}{a^n}\,\frac{(ac,ad;q)_n}{(bc,bd;q)_n}\,
E_{\pm n}(z;a,b,c,d),\quad n\in\ZZ_{\ge0}\,.
\label{K210}
\ee
The notation $E_n$ comes from \cite[\S10.6]{N-St} and
\cite[\S5.3]{Koo-Ma}.

From \eqref{K137}, \eqref{K100} and \eqref{K101}
we can also derive the following duality \cite[Theorem 10.7]{N-St},
\cite[(101)]{Koo-Ma} for non-symmetric AW polynomials.
Put $z_e(n):=eq^n$ ($n\in\ZZ_{\ge0}$) and $z_e(-n):=e^{-1}q^{-n}$
($n\in\ZZ_{>0}$). Then
\be
E_n\big(z_a(m)^{-1};a,b,c,d\big)
=E_m\big(z_{\td a}(n)^{-1};\td a,\td b,\td c,\td d\,\big),\quad m,n\in\ZZ.
\label{K196}
\ee

Put
\be
E_n^-(z;a,b,c,d):=z^{-1}(1-az)(1-bz) E_{n-1}^+(z;qa,qb,c,d),\quad
n\in\ZZ_{>0}\,.
\label{K207}
\ee
By \cite[(97)]{Koo-Ma} we have $E_0(z;a,b,c,d)=1$ and, for $n\in\ZZ_{>0}$\,,
\be
\begin{split}
&E_{-n}(z;a,b,c,d)=E_n^+(z;a,b,c,d)-\frac{(1-q^n ab)(1-q^{n-1}abcd)
\,E_n^-(z;a,b,c,d)}
{q^{n-1}b(1-ab)(1-qab)(1-ac)(1-ad)},\\
&E_n(z;a,b,c,d)=E_n^+(z;a,b,c,d)-\frac{a(1-q^n)(1-q^{n-1}cd)\,E_n^-(z;a,b,c,d)}
{q^{n-1}(1-ab)(1-qab)(1-ac)(1-ad)}.
\end{split}
\label{K205}
\ee

By Proposition 3.2 and (3.16), (3.18), (3.19),
(3.20), (3.21) in \cite{Koo2007}, we can define
\emph{$T_0$-symmetric Askey--Wilson polynomials} $P_n^{\dagger +}(z)$
($n=0,1,2,\ldots$) and
\emph{$T_0$-anti-symmetric Askey--Wilson polynomials}
$P_n^{\dagger -}(z)$ ($n=1,2,\ldots$) by
\begin{equation*}
\begin{split}
&DP_n^{\dagger+}=\la_nP_n^{\dagger+}\;\mbox{and}\;
T_0 P_n^{\dagger+}=-\tfrac{cd}q P_n^{\dagger+}\;\mbox{such that}\;
P_n^{\dagger+}(z)-z^n\in\Span\{z^{-n},\ldots,z^{n-1}\},\\
&DP_n^{\dagger-}=\la_nP_n^{\dagger-}\;\;\mbox{and}\;\;
T_0 P_n^{\dagger-}=-P_n^{\dagger-}\;\;\mbox{such that}\;\;
P_n^{\dagger-}(z)-z^n\in\Span\{z^{-n},\ldots,z^{n-1}\}.
\end{split}
\end{equation*}
By formulas \cite[(3.18), (3.19)]{Koo2007}
we have explicit expressions
\begin{align}
P_n^{\dagger+}(z)&=q^{\half n}\,
P_n^+\big(q^{-\half}z;q^\half a,q^\half b,q^{-\half}c,q^{-\half}d\big),\\
P_n^{\dagger-}(z)&=q^{\half(n-1)}z^{-1}(c-z)(d-z)\,
P_{n-1}^+\big(q^{-1/2}z;q^{1/2} a,q^{1/2} b,q^{1/2} c,q^{1/2} d\big),
\end{align}
and by \cite[(4.1), (4.2)]{Koo2007},
see also \cite[Prop.~5.10(i),(ii)]{N-St},
$P_{\pm n}(z)$ can also be expressed as
a linear combination of $P_n^+(z)$ and $P_n^{\dagger-}(z)$:
\begin{align}
P_{-n}(z)&=\frac1{1-q^{n-1}cd}\,\big(P_n^+(z)-P_n^{\dagger-}(z)\big),\quad
\quad n\in\ZZ_{>0}\,,\label{K191}\\
P_n(z)&=\frac{q^n(1-q^{n-1}abcd)}{1-q^{2n-1}abcd}\,P_n^+(z)+
\frac{1-q^n}{1-q^{2n-1}abcd}\,P_n^{\dagger-}(z),\quad n\in\ZZ_{>0}\,.
\label{K192}.
\end{align}
In \cite[\S4]{Koo2007} \eqref{K191} and \eqref{K192} were first derived and then,
as a corollary, \eqref{K100} and \eqref{K101}.

\subsection{Non-symmetric Askey--Wilson functions}
\begin{df}
The \emph{non-symmetric Askey--Wilson function} $\goE(\ga;z)$
is the rank one case \eqref{K201} of Stokman's normalized Cherednik
kernel \cite[(6.6)]{St2003}:
\begin{align}
\goE(\ga;z)&=
\goE(\ga;z;a,b,c,d):=\frac{(bc,qad^{-1},qbd^{-1},qcd^{-1},qa^{-1}d^{-1};q)_\iy}
{(qabcd^{-1};q)_\iy
(qd^{-1}z,qd^{-1}z^{-1},q\td d^{-1}\ga,q\td d^{-1}\ga^{-1};q)_\iy}\nn\\
&\quad\times\sum_{m=0}^\iy \frac{(-1)^m (ad)^{-m} q^{\half m(m+1)}}
{(1-ab)(1-abcd^{-1})}\,
\frac{(ab,ac,abcd^{-1};q)_m}{(qbd^{-1},qcd^{-1},q;q)_m}\nn\\
&\qquad\times\left(-ab(1-q^m)(1-q^m cd^{-1})\,
E_{-m}(z;a,b,c,qd^{-1})
E_{-m}(\ga;\td a,\td b,\td c,q\td d^{-1})\right.\nn\\
&\qquad\quad\left.+(1-q^mab)(1-q^m abcd^{-1})\,
E_m(z;a,b,c,qd^{-1})
E_m(\ga;\td a,\td b,\td c,q\td d^{-1})\right),
\label{K206}
\end{align}
where the summation is absolutely and uniformly convergent for $(\ga,z)$ in
compacta of $\CC^*\times\CC^*$.
In \eqref{K206} the first term within the big brackets after the summation sign is
supposed to vanish for $m=0$ by the factor $1-q^m$.
\end{df}

\begin{prop}
\label{K208}
\textup{(Stokman, \cite[Theorem 5.17 a), (6.5)]{St2003})}\\
The non-symmetric AW function $\goE(\ga;z)$ is uniquely
characterized by the following properties.\\
\textup{\textbf{(a)}}
The function $G_{q\td d^{-1}}(\ga)^{-1} G_{qd^{-1}}(z)^{-1} \goE(\ga;z)$
depends analytically on $(\ga,z)\in\CC^*\times\CC^*$.\\
\textup{\textbf{(b)}}
$\big(T_1 \goE(\ga;\,\cdot\,)\big)(z)=\big(T_1 \goE(\,\cdot\,;z)\big)(\ga)$.\\
\textup{\textbf{(c)}}
$\td a^{-1} \big(Y \goE(\ga;\,\cdot\,)\big)(z)=\ga^{-1}\goE(\ga;z)$.\\
\textup{\textbf{(d)}}
$a^{-1}\big(Y(\td a,\td b,\td c,\td d) \goE(\,\cdot\,;z)\big)(\ga)=
z^{-1} \goE(\ga;z)$.\\
\textup{\textbf{(e)}}
$\goE(\td a^{-1};a^{-1})=1$.
\end{prop}

For (c) and (d) we took Remark \ref{K145} into account.

\begin{cor}
\label{K209}
\textup{(Stokman, \cite[(6.19), Theorem 6.2 and Theorem 6.9 a)]{St2003})}\\
The non-symmetric AW function $\goE(\ga;z)$ satisfies the
wider normalization formulas
\be
\goE(\ga;a^{-1})=1=\goE(\td a^{-1};z)
\label{K48}
\ee
and the duality
\be
\goE(\ga;z;a,b,c,d)=\goE(z,\ga;\td a,\td b,\td c,\td d),
\label{K50}
\ee
and it reduces for $\ga=(q^n\td a)^{\pm1}$ to non-symmetric AW
polynomials:
\be
\begin{split}
\goE(q^n\td a;z;a,b,c,d)&=E_{-n}(z;a,b,c,d),\quad n\in\ZZ_{>0}\,,\\
\goE(q^{-n}\td a^{-1};z;a,b,c,d)&=E_n(z;a,b,c,d),\quad n\in\ZZ_{\ge0}\,.
\end{split}
\label{K96}
\ee
\end{cor}

Note that the duality \eqref{K50} follows in a straightforward way from
\eqref{K206}.

\begin{theorem}
The non-symmetric AW function can be expressed in terms of
(symmetric) AW functions $\goE^+(\ga;z;a,b,c,d)$ given by
\eqref{K54} or \eqref{K110}:
\begin{align}
&\goE(\ga;z;a,b,c,d)=\goE^+(\ga;z;a,b,c,d)-
\frac{\sqrt{qab^{-1}cd}}{(1-ab)(1-qab)(1-ac)(1-ad)}\nn\\
&\qquad\qquad\qquad\qquad\times \ga^{-1}(1-\td a\ga)(1-\td b\ga)\,
z^{-1}(1-az)(1-bz)\,\goE^+(\ga;z;qa,qb,c,d).
\label{K85}
\end{align}
\end{theorem}

\begin{proof}
Substitute \eqref{K205} four times in the expression in big brackets after the
summation sign in \eqref{K206} and rearrange. We get:
\begin{align}
&-ab(1-q^m)(1-q^m cd^{-1})\,
E_{-m}(z;a,b,c,qd^{-1})E_{-m}(\ga;\td a,\td b,\td c,q\td d^{-1})\nn\\
&\qquad+(1-q^mab)(1-q^m abcd^{-1})\,
E_m(z;a,b,c,qd^{-1})E_m(\ga;\td a,\td b,\td c,q\td d^{-1})\nn\\
&=-ab(1-q^m)(1-q^m cd^{-1})\nn\\
&\qquad\times{\Big(E_m^+(z;a,b,c,qd^{-1})-
\frac{(1-q^m ab)(1-q^m abcd^{-1})}
{q^{m-1}b(1-ab)(1-qab)(1-ac)(1-qad^{-1})}\,E_m^-(z;a,b,c,qd^{-1})\Big)}\nn\\
&\qquad\times\Big(E_{-m}^+(\ga;\td a,\td b,\td c,q\td d^{-1})-
\frac{(1-q^m ab)(1-q^m abcd^{-1})}
{q^{m-1}\td b(1-ab)(1-qab)(1-ac)(1-bc)}\,
E_{-m}^-(\ga;\td a,\td b,\td c,q\td d^{-1})\Big)\nn\\
&+(1-q^mab)(1-q^m abcd^{-1})\nn\\
&\qquad\times\Big(E_m^+(z;a,b,c,qd^{-1})-\frac{a(1-q^m)(1-q^m cd^{-1})}
{q^{m-1}(1-ab)(1-qab)(1-ac)(1-qad^{-1})}\,E_m^-(z;a,b,c,qd^{-1})\Big)\nn\\
&\qquad\times\Big(E_m^+(\ga;\td a,\td b,\td c,q\td d^{-1})
-\frac{\td a(1-q^m)(1-q^m cd^{-1})}
{q^{m-1}(1-ab)(1-qab)(1-ac)(1-bc)}
E_m^-(\ga;\td a,\td b,\td c,q\td d^{-1})\Big)\nn\\
&=(1-ab)(1-q^{2m}abcd^{-1})E_m^+(z;a,b,c,qd^{-1})
E_m^+(\ga;\td a,\td b,\td c,q\td d^{-1})\nn\\
&\qquad
-\frac{\td a(1-q^m)(1-q^m ab)(1-q^m cd^{-1})(1-q^m abcd^{-1})(1-q^{2m}abcd^{-1})}
{q^{2m-2} (1-ab)(1-qab)^2 (1-ac)^2 (1-qad^{-1})(1-bc)}\nn\\
&\qquad\qquad\qquad\qquad
\times E_m^-(z;a,b,c,qd^{-1}) E_m^-(\ga;\td a,\td b,\td c,q\td d^{-1}).
\label{K211}
\end{align}
So $\goE(\ga;z;a,b,c,d)$ can be written as the sum of two terms, namely a term
\begin{align*}
&\frac{(bc,qad^{-1},qbd^{-1},qcd^{-1},qa^{-1}d^{-1};q)_\iy}
{(qabcd^{-1};q)_\iy
(qd^{-1}z,qd^{-1}z^{-1},q\td d^{-1}\ga,q\td d^{-1}\ga^{-1};q)_\iy}\,
\sum_{m=0}^\iy\frac{(-1)^m (ad)^{-m} q^{\half m(m+1)}}
{(1-ab)(1-abcd^{-1})}\,\\
&\qquad\times\frac{(ab,ac,abcd^{-1};q)_m}{(qbd^{-1},qcd^{-1},q;q)_m}\,
(1-ab)(1-q^{2m}abcd^{-1})E_m^+(z;a,b,c,qd^{-1})
E_m^+(\ga;\td a,\td b,\td c,q\td d^{-1}),
\end{align*}
which equals $\goE^+(\ga;z;a,b,c,d)$ by \eqref{K110}, and a term
\begin{align*}
&\goE^-:=-\frac{(bc,qad^{-1},qbd^{-1},qcd^{-1},qa^{-1}d^{-1};q)_\iy}
{(qabcd^{-1};q)_\iy
(qd^{-1}z,qd^{-1}z^{-1},q\td d^{-1}\ga,q\td d^{-1}\ga^{-1};q)_\iy}\,
\sum_{m=1}^\iy\frac{(-1)^m (ad)^{-m} q^{\half m(m+1)}}
{(1-ab)(1-abcd^{-1})}\,\\
&\quad\times\frac{(ab,ac,abcd^{-1};q)_m}{(qbd^{-1},qcd^{-1},q;q)_m}\,
\frac{\td a(1-q^m)(1-q^m ab)(1-q^m cd^{-1})(1-q^m abcd^{-1})(1-q^{2m}abcd^{-1})}
{q^{2m-2} (1-ab)(1-qab)^2 (1-ac)^2 (1-qad^{-1})(1-bc)}\\
&\qquad\qquad\qquad\qquad
\times E_m^-(z;a,b,c,qd^{-1}) E_m^-(\ga;\td a,\td b,\td c,q\td d^{-1}).
\end{align*}
Now observe from \eqref{K110} and \eqref{K207} that
\begin{align*}
&\ga^{-1}(1-\td a\ga)(1-\td b\ga)\,z^{-1}(1-az)(1-bz)\,\goE^+(\ga;z;qa,qb,c,d)\\
&=
\frac{(qbc,q^2ad^{-1},q^2bd^{-1},qcd^{-1},a^{-1}d^{-1};q)_\iy}
{(q^3abcd^{-1};q)_\iy
(qd^{-1}z,qd^{-1}z^{-1},q\td d^{-1}\ga,q\td d^{-1}\ga^{-1};q)_\iy}\,
\sum_{m=1}^\iy (-1)^{m-1} (qad)^{-m} q^{\half m(m-1)}\\
&\qquad\times\frac{1-q^{2m}abcd^{-1}}{1-q^2 abcd^{-1}}\,
\frac{(q^2ab,qac,q^2abcd^{-1};q)_{m-1}}
{(q^2 bd^{-1},qcd^{-1},q;q)_{m-1}}\, E_m^-(z)(z;a,b,c,qd^{-1})\,
E_m^-(\ga;\td a,\td b,\td c,q\td d^{-1}).
\end{align*}
It follows that $\goE^-$ is equal to the second term of the
right-hand side of \eqref{K85}.
\end{proof}

\begin{remark}
\label{K212}
Note that after the last equality in \eqref{K211} the cross terms, i.e.,
the terms with
$E_m^+(z;a,b,c,qd^{-1}) E_m^-(\ga;\td a,\td b,\td c,q\td d^{-1})$ and
$E_m^-(z;a,b,c,qd^{-1}) E_m^+(\ga;\td a,\td b,\td c,q\td d^{-1})$, have
vanished
(just as it should be because of property (b) in Proposition \ref{K208}).
If we try something similar by substituting (normalized versions of)
\eqref{K191} and \eqref{K192} in the first part of \eqref{K211} then
the cross terms do not vanish. In particular, this shows that
$\goF(\ga;z)$, defined in \eqref{K22}, is not equal to $\goE(\ga;z)$ up to
a constant factor.
\end{remark}

Note that properties (a), (b) and (e) in Proposition \ref{K208} and
all formulas in Corollary \ref{K209} also follow from \eqref{K85} together
with properties observed for the symmetric AW function.

As for property (c) (or by duality (d)) in Proposition \ref{K208}
note that, by \eqref{K96} and \eqref{K13},
formula \eqref{K85} extends \eqref{K205}
(or \eqref{K202}). This suggests an alternative proof of property (c) by
imitating the proof for \eqref{K202} as given in \cite[Theorem 4.1]{Koo2007}.
But the proof there is first given for $P_{\pm n}(z)$ as represented by 
\eqref{K191}, \eqref{K192}, and, using that result, next for $P_{\pm n}(z)$ as
represented by \eqref{K202}. In Appendix B we will give a proof of
property (c) by proving first the eigenvalue equation for a function
extending \eqref{K191}, \eqref{K192} and, using this result, next for the
non-symmetric AW function as given by \eqref{K85}.

\begin{theorem}
The non-symmetric AW function $\goE(\ga;z)$ satisfies symmetries
\begin{align}
\goE(\ga;z;a,b,c,d)&=
\frac{(qad^{-1},qa^{-1}d^{-1};q)_\iy}{(ac,a^{-1}c;q)_\iy}\,
\frac{G_{qd^{-1}}(z)}{G_c(z)}\,
\goE(\ga;z;a,b,qd^{-1},qc^{-1}),\label{K51}\\
\goE(\ga;z;a,b,c,d)
&=\frac{(bc,qa^{-1}d^{-1};q)_\iy}{(ac,qb^{-1}d^{-1};q)_\iy}\,
\frac{G_{q\td d^{-1}}(\ga)}{G_{\td c}(\ga)}\,
\goE(\ga;z;b,a,c,d).
\label{K93}
\end{align}
\end{theorem}
\begin{proof}
These symmetries follow from \eqref{K206} together with \eqref{K210}.
Alternatively, both formulas follow by \eqref{K85} from a corresponding
formula (\eqref{K25} and \eqref{K56}, respectively) for the two
symmetric AW functions occurring in \eqref{K85}.

The symmetries \eqref{K51} and \eqref{K93} also follow from
each other by the duality \eqref{K50} in a similar way as in Remark \ref{K94}.
\end{proof}
\paragraph{Acknowledgments}
We thank Jasper Stokman for help in writing Appendix A. Also his comments on an
earlier version of this paper were very helpful. It led us to the present
thorough revision, where the non-symmetric Askey--Wilson function is defined
as Stokman's rank one Cherednik kernel.
We thank Siddhartha Sahi for pointing us to his paper \cite{Sa2000}.
Finally we thank an anonymous referee for careful reading and useful
suggestions to clarify our text at some places.
\appendix
\section{Stokman's kernel formula for $n=1$}
Stokman's formula \cite[(6.6)]{St2003} for the normalized Cherednik kernel
in case of general rank $n$ is the following expansion, absolutely and
uniformly convergent for $(\ga,z)$ in compacta of $(\CC^*)^n\times(\CC^*)^n$.
\be
\goE(\ga,z)=G_{\si\tau}(z)G_\tau(z)
\sum_{s\in S_\tau} \mu_\tau(s) E_{\si\tau}(s;\ga) E_\tau(s;z),
\label{K197}
\ee
where, by \cite[(5.13), (5.15), (6.7)]{St2003},
\be
\mu_\tau(s)=(C_0)_\tau\,\frac{G_{\tau\si\tau}(s)N_{\tau\si}(s^{-1})}
{G_{\tau\si\tau}((s_0)_\tau)N_{\tau\si}((s_0)_{\tau\ddagger})}\,,
\quad s\in\FSS_\tau.
\label{K190}
\ee
Here we will make the case $n=1$ of \eqref{K197} explicit, while working with
parameters $a,b,c,d$.

For a good understanding of \eqref{K197}, \eqref{K190} we recapitulate some
notations from \cite{St2003}, already specializing to $n=1$.
First observe by \cite[(4.7)]{St2003} that Stokman's parameters $t_1,u_1,u_0,t_0$ 
correspond with the parameters $a,b,c,d$ as in \eqref{eq:SO1}, \eqref{eq:SO1inv},
but with $k_1,k_0$ replaced by $t_1,t_0$:
\be\label{K118}
\begin{split}
(a,b,c,d)&=(t_1u_1,-t_1u_1^{-1},q^{1/2} t_0u_0,-q^{1/2} t_0u_0^{-1}),\\
(t_1,u_1,u_0,t_0)&=\Big(i\sqrt{ab},-i\sqrt{ab^{-1}},
i\sqrt{cd^{-1}},-i\sqrt{q^{-1}cd}\,\Big).
\end{split}
\ee
By \cite[p.414]{St2003}, assuming that we work with parameters
$t_1,u_1,u_0,t_0$,
the involution $\si$ interchanges the second and fourth
parameter,
the involution $\tau$ interchanges the third and fourth parameter, and the
involution
$\ddagger$ sends all four parameters and $q$ to their inverses,
just as the action on
parameter space of $\si$, $\tau$ and $\eta$ in our \eqref{eq:sig-cal},
\eqref{eq:tau-new} and \eqref{eq:eta-new} (except that we have minus signs for
$\eta$ acting on $t_1,u_1,u_0,t_0$). Equivalently, by
\cite[p.316]{St2002}, assuming that we work with parameters $a,b,c,d$, then
$\si$, $\tau$ and $\ddagger$ act on parameters
$p_1,p_2,p_3,p_4$ (not necessarily equal to $a,b,c,d$) as
\begin{align*}
&\si(p_1,p_2,p_3,p_4)=(\td p_1,\td p_2,\td p_3,\td p_4):=
\Bigg(\sqrt{q^{-1}p_1p_1p_3p_4},
\sqrt{q\frac{p_1p_2}{p_3p_4}},
\sqrt{q\frac{p_1p_3}{p_2p_4}},\sqrt{q\frac{p_1p_4}{p_2p_3}}\,\Bigg),\\
&\tau(p_1,p_2,p_3,p_4)=(p_1,p_2,p_3,qp_4^{-1}),\quad
\ddagger(p_1,p_2,p_3,p_4;q)=(-p_1^{-1},-p_2^{-1},-p_3^{-1},-p_4^{-1};q^{-1}).
\end{align*}
By \cite[p.414]{St2003}, if $H=H^{a,b,c,d}$ is a parameter depending object
and $t$ maps the parameter space bijectively onto itself,
then $H_t:=H^{t^{-1}(a,b,c,d)}$. Hence,
if $s$ and $t$ are two such mappings then
$H_{st}=H^{(t^{-1}\circ s^{-1})(a,b,c,d)}$. In particular, if $s_1,\ldots,s_r$
are involutions then $H_{s_1\ldots s_r}=
H^{(s_r\ldots s_1)(a,b,c,d)}$. For instance,
$H_{\si\tau}=H^{\td a,\td b,\td c,q\td d^{-1}}$.

By \cite[(6.10)]{St2003} Stokman's spectrum $\FSS$ is the spectrum of
$\td a^{-1}Y$ (with $Y$ as in the present paper) acting in the basic
representation. Thus
by \eqref{K193}, \eqref{K195}, \eqref{K210} the spectral value
$s=q^{-m}\td a^{-1}$ corresponds to the eigenfunction $E_{-m}$ ($m\in\ZZ_{>0}$)
and $s=q^m \td a$ to $E_m$ ($m\in\ZZ_{\ge0}$),
where $E_n$ ($n\in\ZZ$) is the renormalized non-symmetric AW polynomial given by
\eqref{K202}. By \cite[p.425]{St2003} $s_0=\td a$
(the eigenvalue corresponding to eigenfunction $E_0=1$).

Thus \eqref{K197} takes for $n=1$ the form
\begin{align}
&\goE(\ga,z)=(C_0)_\tau\,G_{\si\tau}(\ga) G_\tau(z)
\Bigg(\sum_{m=1}^\iy (C_0)_\tau^{-1}
\mu_\tau\big(q^{-m}(abcd^{-1})^{-1/2}\big)
E_{-m}\big(z;a,b,c,qd^{-1}\big)\nn\\
&\times E_{-m}\big(\ga;\td a,\td b,\td c,q\td d^{-1}\big)
+\sum_{m=0}^\iy (C_0)_\tau^{-1}\mu_\tau\big(q^m(abcd^{-1})^{1/2}\big)
E_m\big(z;a,b,c,qd^{-1}\big)
E_m\big(\ga;\td a,\td b,\td c,q\td d^{-1}\big)\Bigg).
\label{K188}
\end{align}

By \cite[Definition 2.9]{St2003} $G(z)=G_d(z)$ as defined in \eqref{K164}.
Hence
\[
G_\tau(z)G_{\si\tau}(\gamma)=G_{qd^{-1}}(z) G_{q\td d^{-1}}(\ga)=
\frac1{(qd^{-1}z,qd^{-1}z^{-1},q\td d^{-1}\ga,q\td d^{-1}\ga^{-1};q)_\iy}\,.
\]
By \cite[(6.9)]{St2003}
\[
C_0=\frac{(ad,bd,cd,bc,a^{-1}d;q)_\iy}{(abcd;q)_\iy}\,,\quad
(C_0)_\tau=\frac{(qad^{-1},qbd^{-1},qcd^{-1},bc,qa^{-1}d^{-1};q)_\iy}
{(qabcd^{-1};q)_\iy}\,.
\]
Hence
\be
(C_0)_\tau\,G_{\si\tau}(\gamma) G_\tau(z)=
\frac{(bc,qad^{-1},qbd^{-1},qcd^{-1},qa^{-1}d^{-1};q)_\iy}
{(qabcd^{-1};q)_\iy
(qd^{-1}z,qd^{-1}z^{-1},q\td d^{-1}\ga,q\td d^{-1}\ga^{-1};q)_\iy}\,.
\label{K198}
\ee

By \eqref{K190} and \cite[(4.16)]{St2003}
\[
\frac{\mu_\tau\big(\big(q^m(abcd^{-1})^{1/2}\big)^{\pm1}\big)}
{(C_0)_\tau}=
\frac{G_{\tau\si\tau}\big(q^m(abcd^{-1})^{1/2}\big)}
{G_{\tau\si\tau}\big((abcd^{-1})^{1/2}\big)}\,
\frac{N^+_{\tau\si}\big(q^m(abcd^{-1})^{1/2}\big)}
{N^+_{\tau\si}\big((abcd^{-1})^{1/2}\big)}\,
\frac{\FSC_{\tau\si}\big(\big(q^m(abcd^{-1})^{1/2}\big)^{\mp1}\big)}
{\FSC_{\tau\si}\big((abcd^{-1})^{-1/2}\big)}\,.
\]
We have
\[
G_{\tau\si\tau}(z)=G_{\sqrt{a^{-1}bcd}}(z)=
\frac1{\big(\sqrt{a^{-1}bcd}\;z,\sqrt{a^{-1}bcd}\;z^{-1};q\big)_\iy}\,.
\]
Hence
\[
\frac{G_{\tau\si\tau}\big(q^m(abcd^{-1})^{1/2}\big)}
{G_{\tau\si\tau}\big((abcd^{-1})^{1/2}\big)}=
\frac{(bc;q)_m}{(qad^{-1};q)_m}\,(-ad^{-1})^m q^{\half m(m+1)}
\]
(see also this evaluation in the Proof of Theorem 4.2 in \cite{St2002}).
We also have
\[
\frac{N^+_{\tau\si}\big(q^m(abcd^{-1})^{1/2}\big)}
{N^+_{\tau\si}\big((abcd^{-1})^{1/2}\big)}\,
=a^{-2m}\,\frac{1-q^{2m}abcd^{-1}}{1-abcd^{-1}}\,
\frac{(ab,ac,qad^{-1},abcd^{-1};q)_m}{(bc,qbd^{-1},qcd^{-1},q;q)_m}
\]
(see \cite[(4.17)]{St2003} or see this evaluation
in the Proof of Theorem 4.2 in \cite{St2002} which uses \cite[(2.6)]{St2002}).
Thus
\begin{align}
&\frac{G_{\tau\si\tau}\big(q^m(abcd^{-1})^{1/2}\big)}
{G_{\tau\si\tau}\big((abcd^{-1})^{1/2}\big)}\,
\frac{N^+_{\tau\si}\big(q^m(abcd^{-1})^{1/2}\big)}
{N^+_{\tau\si}\big((abcd^{-1})^{1/2})\big)}\nn\\
&\qquad\qquad=(-1)^m (ad)^{-m} q^{{1/2} m(m+1)}\,
\frac{1-q^{2m}abcd^{-1}}{1-abcd^{-1}}\,
\frac{(ab,ac,abcd^{-1};q)_m}{(qbd^{-1},qcd^{-1},q;q)_m}\,.
\label{199}
\end{align}
By \cite[(2.3), (4.14)]{St2003} we have
\[
\FSC(x)=\frac{(1-ax^{-1})(1-bx^{-1})}{1-x^{-2}}
\]
(see also the expression for $\al(x)$ in \cite[\S6.2]{N-St}).
Hence
\begin{align*}
&\FSC_{\tau\si}(x)=\frac{(1-\sqrt{abcd^{-1}}\,x^{-1})(1-\sqrt{abc^{-1}d}\,x^{-1})}
{1-x^{-2}}\,,\\
&\FSC_{\tau\si}\Big(\big(q^m(abcd^{-1})^{1/2}\big)^{-1}\Big)=
\frac{(1-q^mab)(1-q^m abcd^{-1})}{1-q^{2m}abcd^{-1}}\,,\quad
\FSC_{\tau\si}\big((abcd^{-1})^{-1/2}\big)=1-ab,\\
&\FSC_{\tau\si}\Big(q^m(abcd^{-1})^{1/2}\Big)=
-\frac{ab(1-q^m)(1-q^m cd^{-1})}{1-q^{2m}abcd^{-1}}\,.
\end{align*}
So
\be
\begin{split}
&\frac{\FSC_{\tau\si}\big(q^m(abcd^{-1})^{1/2}\big)}
{\FSC_{\tau\si}\big((abcd^{-1})^{-1/2}\big)}
=-\frac{ab(1-q^m)(1-q^m cd^{-1})}{(1-ab)(1-q^{2m}abcd^{-1})}\,,\\
&\frac{\FSC_{\tau\si}\big(q^m(abcd^{-1})^{-1/2}\big)}
{\FSC_{\tau\si}\big((abcd^{-1})^{-1/2}\big)}=
\frac{(1-q^mab)(1-q^m abcd^{-1})}{(1-ab)(1-q^{2m}abcd^{-1})}\,.
\end{split}
\label{K200}
\ee
Altogether, by \eqref{K188}--\eqref{K200} we obtain
\begin{align*}
\goE(\ga,z)&=\frac{(bc,qad^{-1},qbd^{-1},qcd^{-1},qa^{-1}d^{-1};q)_\iy}
{(qabcd^{-1};q)_\iy
(qd^{-1}z,qd^{-1}z^{-1},q\td d^{-1}\ga,q\td d^{-1}\ga^{-1};q)_\iy}\\
&\quad\times\sum_{m=0}^\iy (-1)^m (ad)^{-m} q^{\half m(m+1)}\,
\frac{1-q^{2m}abcd^{-1}}{1-abcd^{-1}}\,
\frac{(ab,ac,abcd^{-1};q)_m}{(qbd^{-1},qcd^{-1},q;q)_m}\\
&\qquad\times\left(-\frac{ab(1-q^m)(1-q^m cd^{-1})}{(1-ab)(1-q^{2m}abcd^{-1})}\,
E_{-m}(z;a,b,c,qd^{-1})
E_{-m}(\ga;\td a,\td b,\td c,q\td d^{-1})\right.\nn\\
&\qquad\quad\left.+\frac{(1-q^mab)(1-q^m abcd^{-1})}{(1-ab)(1-q^{2m}abcd^{-1})}\,
E_m(z;a,b,c,qd^{-1})E_m(\ga;\td a,\td b,\td c,q\td d^{-1})\right).
\end{align*}
Hence
\begin{align}
\goE(\ga,z)&=\frac{(bc,qad^{-1},qbd^{-1},qcd^{-1},qa^{-1}d^{-1};q)_\iy}
{(qabcd^{-1};q)_\iy
(qd^{-1}z,qd^{-1}z^{-1},q\td d^{-1}\ga,q\td d^{-1}\ga^{-1};q)_\iy}\nn\\
&\quad\times\sum_{m=0}^\iy \frac{(-1)^m (ad)^{-m} q^{\half m(m+1)}}
{(1-ab)(1-abcd^{-1})}\,
\frac{(ab,ac,abcd^{-1};q)_m}{(qbd^{-1},qcd^{-1},q;q)_m}\nn\\
&\qquad\times\left(-ab(1-q^m)(1-q^m cd^{-1})\,
E_{-m}(z;a,b,c,qd^{-1})
E_{-m}(\ga;\td a,\td b,\td c,q\td d^{-1})\right.\nn\\
&\qquad\quad\left.+(1-q^mab)(1-q^m abcd^{-1})\,
E_m(z;a,b,c,qd^{-1})
E_m(\ga;\td a,\td b,\td c,q\td d^{-1})\right),
\label{K201}
\end{align}
which is an explicit expression for the rank one normalized Cherednik kernel,
with the summation absolutely and uniformly convergent in compacta of
$\CC^*\times\CC^*$.
\section{An alternative proof that the non-symmetric Askey--Wilson function
is an eigenfunction of $Y$}
In the proof of \eqref{K193} in \cite[Theorem 4.1]{Koo2007} a first step was
the following
$q$-difference formula \cite[(4.8)]{Koo2007} for AW polynomials.
\be
\de_{q,z}\,P_n^+(z;a,b,c,d)
=(q^{\half n}-q^{-\half n})(z-z^{-1})
P_{n-1}^+(z;q^{1/2} a,q^{1/2} b,q^{1/2} c,q^{1/2} d),
\label{K17}
\ee
where $P_n^+(z)$ is the monic AW polynomial given by \eqref{K14}
and where
\[
\de_{q,z}\,f(z):=f(q^{1/2} z)-f(q^{-1/2} z).
\]
We will extend \eqref{K17} for AW functions as given by \eqref{K80},
\begin{lm}
We have
\be
\de_{q,z}\,\goE^+(\ga;z;a,b,c,d)
=\frac{q^{1/2} a(1-\ga\td a)(1-\ga^{-1}\td a)}
{(1-ab)(1-ac)(1-ad)}\,
(z-z^{-1})
\goE^+(\ga;z;q^{1/2} a,q^{1/2} b,q^{1/2} c,q^{1/2} d).
\label{K10}
\ee
\end{lm}
\begin{proof}
Straightforward computations give
\[
\de_{q,z}\,(az,az^{-1};q)_k
=q^{-1/2}a(1-q^k)(z-z^{-1})(q^{1/2} az,q^{1/2} az^{-1};q)_{k-1}
\]
and
\begin{multline*}
\de_{q,z}\,\frac{(az,az^{-1};q)_\iy}
{(q^{k+1}d^{-1}z,q^{k+1}d^{-1}z^{-1};q)_\iy}\\
=q^{-1/2}a(1-q^{k+1}a^{-1}d^{-1})(z-z^{-1})\,
\frac{(q^{1/2} az,q^{1/2} az^{-1};q)_\iy}
{(q^{k+\half}d^{-1}z,q^{k+\half}d^{-1}z^{-1};q)_\iy}\,.
\end{multline*}
Hence, by \eqref{K81},
\be
\de_{q,z}\,\goR(\ga;z;a,b,c,d)
=\frac{q^{1/2} a(1-\ga\td a)(1-\ga^{-1}\td a)}
{(1-ab)(1-ac)(1-ad)}\,
(z-z^{-1})
\goR(\ga;z;q^{1/2} a,q^{1/2} b,q^{1/2} c,q^{1/2} d)
\label{K82}
\ee
and
\begin{multline*}
\de_{q,z}\,\frac{(az,az^{-1};q)_\iy}{(qd^{-1}z,qd^{-1}z^{-1};q)_\iy}\,
\goR(\ga;z;qd^{-1},b,c,qa^{-1})=q^{-1/2}a(1-qa^{-1}d^{-1})(z-z^{-1})\\
\times\frac{(q^{1/2} az,q^{1/2} az^{-1};q)_\iy}
{(q^{1/2} d^{-1}z,q^{1/2} d^{-1}z^{-1};q)_\iy}\,
\goR(\ga;z;q^{1/2} d^{-1},q^{1/2} b,q^{1/2} c,q^{1/2} a^{-1}).
\end{multline*}
We conclude that \eqref{K82} remains true if there
$\goR(\ga;z;a,b,c,d)$ is replaced by the second term on the right in
\eqref{K80}, and hence also if there $\goR(\ga;z;a,b,c,d)$ is replaced by
$\goE^+(\ga;z;a,b,c,d)$.
\end{proof}

Note that, for $\ga=q^n\td a$, \eqref{K10} reduces by \eqref{K14} and
\eqref{K13} to \eqref{K17}.
\sLP\indent
Another important element in the proof of \cite[Theorem 4.1]{Koo2007} is the
following lemma (only sketched there).

\begin{lm}\label{K88}
If $f\in\FSO_{qd^{-1}}^+$ then
\begin{align}
(Yf)(z)&=
\frac{(c-z)(d-z)(1+ab-(a+b)z)}{(1-z^2)(q-z^2)}\,(f(q^{-1}z)-f(z))\nn\\
&\qquad+\frac{(1-az)(1-bz)(1-cz)(1-dz)}{(1-z^2)(1-qz^2)}\,(f(qz)-f(z))
+q^{-1}abcd\,f(z).\label{K11}
\end{align}
Also, if $g(z)=z^{-1}(c-z)(d-z)\,h(q^{-1/2}z)$ with
$h\in\FSO_{qd^{-1}}^+$ then
\begin{multline}
(Yg)(z)=\frac{(c-z)(d-z)(1+ab-(a+b)z)}{z(1-z^2)}\,h(q^{-1/2} z)\\
-\frac{(1-az)(1-bz)(1-cz)(1-dz)}{z(1-z^2)}\,h(q^{1/2} z).
\label{K12}
\end{multline}
\end{lm}
\begin{proof}
Recall from \cite[(3.13)]{Koo2007} that $(Yf)(z)$ equals
\begin{align*}
&\frac{z \bigl(1+ab-(a+b)z\bigr)
\bigl((c+d)q-(cd+q)z\bigr)\,f(z)}{q(1-z^2)(q-z^2)}
+\frac{(1-az)(1-bz)(1-cz)(1-dz)f(qz)}{(1-z^2)(1-q z^2)}\\
&+\frac{(1-a z)(1-b z) \bigl((c+d)qz-(cd+q)\bigr)f(z^{-1})}
{q(1-z^2)(1-q z^2)}
+\frac{(c-z)(d-z)\bigl(1+ab-(a+b)z\bigr)f(qz^{-1})}{(1-z^2)(q-z^2)}.
\end{align*}
Then \eqref{K11} and \eqref{K12} follow by straightforward computations.
\end{proof}

\begin{prop}
Define a non-symmetric AW type function
\begin{multline}
\goF(\ga;z)=
\goF(\ga;z;a,b,c,d):=\goE^+(\ga;z;a,b,c,d)-\frac{a(1-\td a\ga)}{(1-ab)(1-ac)(1-ad)}\\
\times z^{-1}(c-z)(d-z)
\goE^+(\ga;q^{-1/2}z;q^{1/2} a,q^{1/2} b,q^{1/2} c,q^{1/2} d).
\label{K22}
\end{multline}
It satisfies the eigenvalue equation
\be
\td a^{-1}(Y\, \goF(\ga;\,\cdot\,;a,b,c,d))(z)=\ga^{-1} \goF(\ga;z;a,b,c,d).
\label{K23}
\ee
\end{prop}
\begin{proof}
Consider \eqref{K11} and \eqref{K12} with
\[
f(z)=\goE^+(\ga;z;a,b,c,d),\quad
h(z)=\goE^+(\ga;z;q^{1/2} a,q^{1/2} b,q^{1/2} c,q^{1/2} d),
\]
and $g(z)=z^{-1}(c-z)(d-z)h(q^{-1/2}z)$. Then
\begin{align*}
(Yf)(z)-\td a\ga^{-1} f(z)&=
\frac{(c-z)(d-z)(1+ab-(a+b)z)}{(1-z^2)(q-z^2)}\,(f(q^{-1}z)-f(z))\nn\\
&\qquad+\frac{(1-az)(1-bz)(1-cz)(1-dz)}{(1-z^2)(1-qz^2)}\,(f(qz)-f(z))
+(\td a^2-\td a\ga^{-1})f(z)
\end{align*}
and
\begin{align*}
&(Yg)(z)-\td a\ga^{-1} g(z)\\
&=\frac{(c-z)(d-z)}{q-z^2}\left(\frac{(1+ab-(a+b)z)}{z(1-z^2)}-\td a\ga^{-1}\right)
\,\frac{(1-ab)(1-ac)(1-ad)}{a(1-\ga\td a)(1-\ga^{-1}\td a)}\,
(f(q^{-1}z)-f(z))\\
&\qquad+\frac{(1-az)(1-bz)(1-cz)(1-dz)}{(1-z^2)(1-qz^2)}\,
\frac{(1-ab)(1-ac)(1-ad)}{a(1-\ga\td a)(1-\ga^{-1}\td a)}\,
(f(qz)-f(z)),
\end{align*}
where we also used \eqref{K10}.
Now take a linear combination of the last two equalities as in \eqref{K22}.
Then we obtain
\begin{align*}
&(Y\, \goF(\ga;\,\cdot\,;a,b,c,d))(z)-\td a\ga^{-1} \goF(\ga;z;a,b,c,d)\\
&=\frac1{1-\td a^{-1}\ga}\left(
\frac{(1-az)(1-bz)(1-cz)(1-dz)}{(1-z^2)(1-qz^2)}\,
\bigl(f(qz)-f(z)\bigr)\right.\\
&\qquad\left.+\frac{(a-z)(b-z)(c-z)(d-z)}{(1-z^2)(q-z^2)}\,
\bigl(f(q^{-1}z)-f(z)\bigr)+(1-\td a\ga)(1-\td a\ga^{-1})f(z)\right),
\end{align*}
which vanishes because of \eqref{K72} and \eqref{K18}.
\end{proof}

By \eqref{K13} and \eqref{K14}
formula \eqref{K22} reduces for $\ga=q^n\td a$ to
\begin{align*}
&\goF(q^n\td a;z;a,b,c,d)=
E_n^+(z;a,b,c,d)
-\frac{a(1-q^{n-1}abcd)}{(1-ab)(1-ac)(1-ad)}\,
z^{-1}(c-z)(d-z)\\
&\times E_{n-1}^+(q^{-1/2}z;q^{1/2} a,q^{1/2} b,q^{1/2} c,q^{1/2} d)
=\frac{(q^{n-1}abcd;q)_n a^n}{(ab,ac,ad;q)_n}\\
&\times\left(P_n^+(z;a,b,c,d)
-q^{\half(n-1)}z^{-1}(c-z)(d-z)
P_{n-1}^+(q^{-1/2}z;q^{1/2} a,q^{1/2} b,q^{1/2} c,q^{1/2} d)\right),
\end{align*}
which is a constant multiple of the non-symmetric AW polynomial
$P_{-n}(z)$ as given in \eqref{K191}.
Also, for $\ga=q^{-n}\td a^{-1}$, formula \eqref{K22} reduces to
\begin{align*}
&\goF(q^{-n}\td a^{-1};z;a,b,c,d)=
E_n^+(z;a,b,c,d)
-\frac{a(1-q^{-n})}{(1-ab)(1-ac)(1-ad)}\,
z^{-1}(c-z)(d-z)\\
&\times E_{n-1}^+(q^{-1/2}z;q^{1/2} a,q^{1/2} b,q^{1/2} c,q^{1/2} d)
=\frac{(q^{n-1}abcd;q)_n a^n}{(ab,ac,ad;q)_n}\\
&\times\left(P_n^+(z;a,b,c,d)
-\frac{q^{\half(n-1)}(1-q^{-n})}{1-q^{n-1}abcd}\,z^{-1}(c-z)(d-z)
P_{n-1}^+(q^{-1/2}z;q^{1/2} a,q^{1/2} b,q^{1/2} c,q^{1/2} d)\right),
\end{align*}
which is a constant multiple of $P_n(z)$ as given in \eqref{K192}.
For these two choices of $\ga$ the eigenvalue equation \eqref{K23} also agrees
with the eigenvalue equations \eqref{K193}, \eqref{K195}

As observed in Remark \ref{K212}, $\goF(\ga;z)$ cannot be equal
(up to a constant factor) to the non-symmetric AW function
$\goE(\ga;z)$. Furthermore, from \eqref{K22} we cannot
read off whether $\goF(\ga;z)$ satisfies a duality and a $c\leftrightarrow qd^{-1}$
symmetry as in \eqref{K50} and \eqref{K51}.

We now continue with the expression \eqref{K85} for $\goE(\ga;z)$,
which extends the expression for the non-symmetric AW polynomials
given in \eqref{K100}, \eqref{K101}. We will prove the eigenfunction result
from this expression, but we need two preparatory lemmas.

\begin{lm}
We have
\begin{align}
&z(1-q^{-1/2}az^{-1})(1-q^{-1/2}bz^{-1})\goE^+(\ga;q^{-1/2}z;a,b,c,d)\nn\\
&\qquad\qquad-z^{-1}(1-q^{-1/2}az)(1-q^{-1/2}bz)\goE^+(\ga;q^{1/2} z;a,b,c,d)\nn\\
&\qquad
=(1-q^{-1}ab)(z-z^{-1}) \goE^+(\ga;z;q^{-1/2}a,q^{-1/2}b,q^{1/2} c,q^{1/2} d).
\label{K31}
\end{align}
\end{lm}
\begin{proof}
Straightforward computations give
\begin{align*}
&z(1-q^{-1/2}az^{-1})(1-q^{-1/2}bz^{-1})(q^{-1/2}az,q^{1/2} az^{-1};q)_k\\
&\qquad\qquad
-z^{-1}(1-q^{-1/2}az)(1-q^{-1/2}bz)(q^{1/2} az,q^{-1/2} az^{-1};q)_k\\
&\qquad=(1-q^{k-1}ab)(z-z^{-1}) (q^{-1/2}az,q^{-1/2} az^{-1};q)_k
\end{align*}
and
\begin{multline*}
z(1-q^{-1/2}az^{-1})(1-q^{-1/2}bz^{-1})\,
\frac{(q^{-1/2}az,q^{1/2} az^{-1};q)_\iy}
{(q^{k+\half}d^{-1}z,q^{k+\frac32}d^{-1}z^{-1};q)_\iy}
-z^{-1}(1-q^{-1/2}az)(1-q^{-1/2}bz)\\
\times\frac{(q^{1/2} az,q^{-1/2}az^{-1};q)_\iy}
{(q^{k+\frac32}d^{-1}z,q^{k+\half}d^{-1}z^{-1};q)_\iy}
=(1-q^k bd^{-1})(z-z^{-1})\,
\frac{(q^{-1/2}az,q^{-\half}az^{-1};q)_\iy}
{(q^{k+\half}d^{-1}z,q^{k+\half}d^{-1}z^{-1};q)_\iy}\,.
\end{multline*}
Hence, by \eqref{K81},
\begin{align}
&z(1-q^{-1/2}az^{-1})(1-q^{-1/2}bz^{-1})\goR(\ga;q^{-1/2}z;a,b,c,d)\nn\\
&\qquad\qquad-z^{-1}(1-q^{-1/2}az)(1-q^{-1/2}bz)\goR(\ga;q^{1/2} z;a,b,c,d)\nn\\
&\qquad=(1-q^{-1}ab)(z-z^{-1}) \goR(\ga;z;q^{-1/2}a,q^{-1/2}b,q^{1/2} c,q^{1/2} d)
\label{K84}
\end{align}
and
\begin{multline*}
z(1-q^{-1/2}az^{-1})(1-q^{-1/2}bz^{-1})\,
\frac{(q^{-1/2}az,q^{1/2} az^{-1};q)_\iy}
{(q^{1/2}d^{-1}z,q^{\frac32}d^{-1}z^{-1};q)_\iy}\,
\goR(\ga;q^{-1/2}z;qd^{-1},b,c,qa^{-1})\\
-z^{-1}(1-q^{-1/2}az)(1-q^{-1/2}bz)\,
\frac{(q^{1/2} az,q^{-1/2}az^{-1};q)_\iy}
{(q^{\frac32}d^{-1}z,q^{{1/2}}d^{-1}z^{-1};q)_\iy}\,
\goR(\ga;q^{1/2} z;qd^{-1},b,c,qa^{-1})\\
=(1-bd^{-1})(z-z^{-1})
\goR(\ga;z;q^{1/2} d^{-1},q^{-1/2}b,q^{1/2} c,q^{\frac32}a^{-1}).
\end{multline*}
We conclude that \eqref{K84} remains true if there
$\goR(\ga;z;a,b,c,d)$ is replaced by the second term on the right in
\eqref{K80}, and hence also if there $\goR(\ga;z;a,b,c,d)$ is replaced by
$\goE^+(\ga;z;a,b,c,d)$.
\end{proof}

For $\ga=q^n\td a$ formula \eqref{K31} becomes a formula for 
AW polynomials which was earlier given in \cite[(2.1B)]{Ka-Mi}.

\begin{lm}
If $k\in\FSO_{qd^{-1}}^+$ and $l(z):=z^{-1}(1-az)(1-bz)k(z)$ then 
\begin{align}
&a^{-1}b^{-1}(Yl)(z)
=\left(a+b+\tfrac1q(-\tfrac1a-\tfrac1b+c+d)\right.\nn\\
&\qquad\qquad\qquad\qquad
\left.+abcd\left(z+\tfrac1z-\tfrac1a-\tfrac1b-\tfrac1c-\tfrac1d\right)+
\left(\tfrac1{qab}+cd\left(1-\tfrac1q\right)-1\right)z^{-1}\right)k(z)\nn\\
&\qquad\qquad+\frac{(qa-z)(qb-z)(c-z)(d-z)(1+ab-(a+b)z)}{qabz(1-z^2)(q-z^2)}\,
(k(q^{-1}z)-k(z))\nn\\
&\qquad\qquad+\frac{(1-az)(1-bz)(1-cz)(1-dz)(1-qaz)(1-qbz)}{qabz(1-z^2)(1-qz^2)}\,
(k(qz)-k(z)).\label{K87}
\end{align}
\end{lm}
\begin{proof}
This follows by straightforward computation from the expression for $(Yf)(z)$
in the Proof of Lemma \ref{K88}.
\end{proof}

\paragraph{Proof of Proposition \ref{K208} (c)}\quad\\
We have to prove that
\be
\td a^{-1}(Y\,\goE(\ga;\,\cdot,;a,b,c,d))(z)=\ga^{-1} \goE(\ga;z;a,b,c,d).
\label{K47}
\ee

Consider \eqref{K22} and \eqref{K85} with
\begin{align*}
&f(z)=\goE^+(\ga;z;a,b,c,d),\quad
g(z)=z^{-1}(c-z)(d-z)h(q^{-1/2}z),\\
&h(z)=\goE^+(\ga;z;q^{1/2} a,q^{1/2} b,q^{1/2} c,q^{1/2} d),\\
&k(z):=\goE^+(\ga;z;qa,qb,c,d),\quad
l(z):=z^{-1}(1-az)(1-bz)k(z).
\end{align*}
On comparing \eqref{K23} and \eqref{K47} we see that it is sufficient to prove that
\be
((Y-\td a\ga^{-1})g)(z)
=\frac d{1-qab}\,\frac{1-\ga\sqrt{qabc^{-1}d^{-1}}}{\ga\sqrt{q^{-1}abc^{-1}d}}\,
((Y-\td a\ga^{-1})l)(z).
\label{K89}
\ee
On the left-hand side of \eqref{K89} substitute \eqref{K12} and next substitute
in the resulting expression for $h(q^{\pm1/2}z)$ the following
formulas obtained from
\eqref{K31}:
\begin{align*}
h(q^{-1/2}z)&=
\frac{(1-qaz^{-1})(1-qbz^{-1})}{(1-qab)(1-qz^{-2})}\,k(q^{-1}z)
+\frac{(1-az)(1-bz)}{(1-qab)(1-q^{-1}z^2)}\,k(z),\\
h(q^{1/2}z)&=
\frac{(1-az^{-1})(1-bz^{-1})}{(1-qab)(1-q^{-1}z^{-2})}\,k(z)
+\frac{(1-qaz)(1-qbz)}{(1-qab)(1-qz^2)}\,k(qz).
\end{align*}
On the right-hand side of \eqref{K89} substitute \eqref{K87}.
From the resulting equality \eqref{K80}, which yet has to be proved, consider
the left-hand side minus the right-hand side. We can show that this equals
\begin{align*}
&\frac{cd(1-az)(1-bz)}{q\td a}(qab-1)\ga z\left(
\frac{(1-qaz)(1-qbz)(1-cz)(1-dz)}{(1-z^2)(1-qz^2)}\,
\bigl(k(qz)-k(z)\bigr)\right.\\
&\qquad\left.+\frac{(qa-z)(qb-z)(c-z)(d-z)}{(1-z^2)(q-z^2)}\,
\bigl(k(q^{-1}z)-k(z)\bigr)+(1-q\td a\ga)(1-q\td a\ga^{-1})k(z)\right),
\end{align*}
which vanishes because of \eqref{K72} and \eqref{K18}.\qed

\end{document}